\def\cy{C^*\hspace{-.5mm}(y)}

\def\suc{\mbox{Suc}}\def\Th{\mbox{Th}}
\def\TC{\mbox{TC}}
\def\stem{\mbox{stem}}
\def\zfc{\mbox{{ZFC}}}
\def\aaq{\mbox{\tt a\hspace{-0.4pt}a}}

\def\MM{\mbox{\rm MM}}

scaled \magstephalf 

\def\d{\delta}

\def\om{\omega}

\def\restriction{\upharpoonright}

\def\cof{\mathop{\rm cf}}

\def\len{\mathop{\rm len}}

\def\dom{\mathop{\rm dom}}

\def\P{{\cal P}}
\def\On{{\rm On}}
\def\la{\langle}
\def\ra{\rangle}

\def\Q{{\mathbb Q}}
\def\R{{\mathbb R}}

\def\mm{{\mathcal M}}
\def\mn{{\mathcal N}}

\newcommand{\gl}{\lambda}

\newcommand{\K}{{\cal K}}

\renewcommand{\t}{{{\cal T}_{\gl,\gk}}}
\newcommand{\U}{{\cal U}}

\renewcommand{\P}{{\cal P}}

\newcommand{\cf}{\mathop{\rm cf}}

\mathchardef\bfSigma="0606 \mathchardef\bfPi="0605
\mathchardef\bfDelta="0601

\newcommand{\force}{\Vdash} 

\newcommand{\open}{\Bbb}

\newcommand{\oC}{{\open C}}

\newcommand{\oP}{{\open P}}
\newcommand{\oQ}{{\open Q}}
\newcommand{\oR}{{\open R}}

\newcommand{\ben}{\begin{enumerate}}
\newcommand{\een}{\end{enumerate}}
\newcommand{\bd}{\begin{definition}}
\newcommand{\ed}{\end{definition}}
\newcommand{\bl}{\begin{lemma}}
\newcommand{\el}{\end{lemma}}
\newcommand{\bt}{\begin{theorem}}
\newcommand{\et}{\end{theorem}}
\newcommand{\bp}{\begin{proposition}}
\newcommand{\ep}{\end{proposition}}

\def\fol{\mbox{FO}}

\def\looo{{\L}_{\omega_1\omega}}
\def\lio{\L_{\infty\omega}}
\def\loo{\L_{\omega\omega}}

\makeatletter

\def\bracketdown#1{\mathop{\vbox{\ialign{##\crcr\noalign{\kern2\p@}
\downbracketfill\crcr\noalign{\kern2\p@\nointerlineskip}
$\hfil\displaystyle{#1}\hfil$\crcr}}}\limits}

\def\bracketup#1{\mathop{\vbox{\ialign{##\crcr\noalign{\kern1\p@}
\upbracketfill\crcr\noalign{\kern1\p@\nointerlineskip}
$\hfil\displaystyle{#1}\hfil$\crcr}}}\limits}

\def\upbracketfill{$\m@th
\makesm@sh{\llap{\vrule\@height2\p@\@width.4\p@}}%
\leaders\vrule\@height.4\p@\hfill
\makesm@sh{\rlap{\vrule\@height2\p@\@width.4\p@}}$}

\def\downbracketfill{$\m@th
\makesm@sh{\llap{\vrule\@height.4\p@\@depth1.6\p@\@width.4\p@}}%
\leaders\vrule\@height.4\p@\hfill
\makesm@sh{\rlap{\vrule\@height.4\p@\@depth1.6\p@\@width.4\p@}}$}
\makeatother

\newcommand{\Loo}{\L_{\omega_1\omega}}

\def\\(\two\){Eloise}

\def\x{\xi}
\def\a{\alpha}

\def\g{\gamma}
\def\b{\beta}

\def\two{\mbox{$\mathbf{II}$}}

\def\DEF{\mbox{Def}}
\def\MMa{\mbox{\tiny MM}}
\def\phi{\varphi}
\def\hod{\mbox{\rm HOD}}
\def\sol{\mbox{$\L^2$}}
\newcommand\cofmodel[1]{C^*_{#1}}
\def\t{\tau}

\documentclass[12pt]{article} 
\usepackage{amsthm,hyperref} 
\usepackage{graphicx} 
\usepackage{amssymb,latexsym,times}
\newcommand{\psfrag}[2]{}
\usepackage{makeidx,proof}
\DeclareMathAlphabet{\mathpzc}{OT1}{pzc}{m}{it}
\makeindex   
 \theoremstyle{plain}
  \newtheorem{theorem}{Theorem}[section]
  \newtheorem{lemma}[theorem]{Lemma}
  \newtheorem*{corollary}{Corollary}
  \newtheorem{proposition}[theorem]{Proposition}

  \theoremstyle{definition}
  \newtheorem{definition}[theorem]{Definition}
  \newtheorem{example}[theorem]{Example}

  \theoremstyle{remark}
  \newtheorem*{remark}{Remark}
  
 \newtheorem*{claim}{Claim}

\def\L{\mathcal{L}}
\begin{document}

\author{Juliette Kennedy\thanks{Research partially supported by
grant 322488 of the Academy of Finland.}\\ Helsinki \and Menachem Magidor\thanks{Research supported by the Simons Foundation and  the Israel Science Foundation grant 817/11.}\\ Jerusalem \and Jouko V\"a\"an\"anen\thanks{Research supported by the Simons Foundation and
grant 322795 of the Academy of Finland.}\\ Helsinki and Amsterdam}

\title{Inner Models from Extended Logics: Part 1\thanks{The authors would like to thank the Isaac Newton Institute for Mathematical Sciences for its hospitality during the programme Mathematical, Foundational and Computational Aspects of the Higher Infinite supported by EPSRC Grant Number EP/K032208/1.  The authors are grateful to John Steel, Philip Welch and Hugh Woodin for comments on the results presented here.}}
  
\maketitle

{
\def\l{\lambda}

\begin{abstract}
If we replace first order logic by second order logic in the original definition of G\"odel's inner model $L$, we obtain \hod\ (\cite{MR0281603}). In this paper we consider inner models that arise if we replace first order logic by a logic that has some, but not all, of the strength of second order logic. Typical examples are the extensions of first order logic by generalized quantifiers, such as  the Magidor-Malitz quantifier (\cite{MR0453484}), the cofinality quantifier (\cite{MR0376334}), or stationary logic (\cite{MR486629}). Our first set of results show that both $L$ and \hod\ manifest
some amount of  {\em formalism freeness} in the sense that they are not very sensitive to the choice of the underlying logic. Our second set of results shows that the cofinality quantifier  gives rise to a new robust inner model between $L$ and \hod.  We show, among other things, that assuming a proper class of Woodin cardinals the regular cardinals $>\aleph_1$ of $V$ are weakly compact in the inner model arising from the cofinality quantifier and the theory of that model is (set) forcing absolute and independent of the cofinality in question. We do not know whether this model satisfies the Continuum Hypothesis, assuming large cardinals, but we can show, assuming three Woodin cardinals and a measurable above them,  that if the construction is relativized to a real, then on a cone of reals the Continuum Hypothesis is true in the relativized model.
\end{abstract}


\section{Introduction}


Inner models, together with the forcing method, are the basic building blocks used by set theorists to prove relative consistency results on the one hand and to try to chart the ``true" universe of set theory $V$ on the other hand. 

The first and best known, also the smallest of the inner models is G\"odel's $L$, the universe of constructible sets. An important landmark among the largest inner models is the universe of hereditarily ordinal definable sets $\hod$, also introduced by G\"odel\footnote{G\"odel introduced $\hod$ in his 1946 {\em Remarks before the Princeton Bicenntenial conference on problems in mathematics} \cite{bicentennial}. The lecture was  given during a session on computability organized by Alfred Tarski, and in it 
G\"odel asks whether notions of definability and provability can be isolated in the set-theoretic formalism, which admit a form of robustness similar to that exhibited by the notion of general recursiveness:
``Tarski has stressed in his lecture the great importance (and I think justly) of the concept of general recursiveness (or Turing computability). It seems to me that this importance is largely due to the fact that with this concept one has succeeded in giving an absolute definition of an interesting epistemological notion, i.e. one not depending on the formalism chosen. In all other cases treated previously, such as definability or demonstrability, one has been able to define them only relative to a given language, and for each individual language it is not clear that the one thus obtained is not the one looked for. For the concept of computability however\ldots the situation is different\ldots This, I think, should encourage one to expect the same thing to be possible also in other cases (such as demonstrability or definability)."  

G\"odel contemplates the idea that constructibility might be a suitable analog of the notion of general recursiveness. 
G\"odel also considers the same for $\hod$, and predicts the consistency of the axiom $V=HOD + 2^{\aleph_0} > \aleph_1$ (proved later by McAloon \cite{MR0292670}). See \cite{MR3134897} for a development of G\"odel's proposal in a ``formalism free" direction.
}. In between these two extremes there is  a variety of inner models arising from enhancing G\"odel's $L$ by  normal ultrafilters on  measurable cardinals, or in a more general case extenders, something that $L$ certainly does not have itself. 

We propose a construction of inner models which arise not from adding normal ultrafilters, or extenders, to $L$, but by changing the underlying  construction of $L$. We show that the new inner models have similar forcing absoluteness properties as $L(\oR)$, but at the same time they satisfy the Axiom of Choice.

 G\"odel's  hierarchy of constructible sets is defined by reference to first order definability. Sets on a higher level are the first order definable sets of elements of lower levels. The inner model $L$ enjoys strong forcing absoluteness: truth in $L$ cannot be changed by forcing, in fact not by any method of extending the universe without adding new ordinals. Accordingly, it is usually possible to settle in $L$, one way or other, any set theoretical question which is otherwise independent of ${\zfc}$. However, the problem with $L$ is that it cannot have large cardinals on the level of the Erd\H os cardinal $\kappa(\omega_1)$ or higher. To remedy this, a variety of inner models, most notably the smallest inner model $L^\mu$ with a measurable cardinal, have been introduced (see e.g. \cite{MR2768698}). 

We investigate the question to what extent is it essential that first order definability is used in the construction of G\"odel's $L$. In particular, what would be the effect of changing first order logic to a stronger logic? In fact  there are two precedents: Scott and Myhill \cite{MR0281603} showed that if first order definability is replaced by second order definability the 
all-encompassing  class $\hod$ of hereditarily ordinal definable sets is obtained. The inner model $L$ is thus certainly sensitive to the definability concept used in its construction. The inner model $\hod$ has consistently even supercompact cardinals \cite{MR0540771}. However, $\hod$ does not solve any of the central independent statements of set theory; in particular, it does not solve the Continuum Hypothesis or the Souslin Hypothesis \cite{MR0292670}.

A second precedent is provided by  Chang \cite{MR0280357} in which first order definability was replaced by definability in the infinitary language $\L_{\omega_1\omega_1}$, obtaining what came to be known as the {\em Chang model}. Kunen \cite{MR0337603} showed that the Chang model fails to satisfy the Axiom of Choice, if the existence of uncountably many  measurable cardinals is assumed. We remark that  the inner model $L(\oR)$ arises in the same way if $\L_{\omega_1\omega}$ is used instead of  
$\L_{\omega_1\omega_1}$. Either way, the resulting inner model fails to satisfy the Axiom of Choice if enough large cardinals are assumed. This puts these inner models in a different category. On the other hand, the importance of both the Chang model and $L(\oR)$ is accentuated by the result of  Woodin \cite{MR959110} that under large cardinal assumptions the first order theory of the Chang model, as well as of $L(\oR)$, is absolute under set forcing. So there would be reasons to expect that these inner models would solve several independent statements of set theory, e.g. the CH. However, the failure of the Axiom of Choice in these inner models dims the light  such ``solutions" would shed on CH. For example, assuming large cardinals, the model $L(\oR)$ satisfies the statement ``Every uncountable set of reals contains a perfect subset", which under AC would be equivalent to $CH$.
On the other hand,  large cardinals imply that there is  in  $L(\oR)$ a surjection from $\oR$ onto $\omega_2$, which under AC would imply $\neg CH$.

In this paper we define analogs of the constructible hierarchy by replacing first order logic in G\"odel's construction by any one of a number of logics. The inner models $\hod$, $L(\oR)$ and the Chang model are special cases, obtained by replacing first order definability by definability in $\L^2$, $\looo$ and $\L_{\omega_1\omega_1}$, respectively. Our main focus is on extensions of first order logic by generalized quantifiers in the sense of Mostowski \cite{MR0089816} and Lindstr\"om \cite{MR0244012}. 
We obtain new inner models which are $L$-like in that they are models of {\zfc} and their theory  is absolute under set forcing, but at the same time these inner models contain large cardinals, or inner models with large cardinals. 
 
The resulting inner models  enable us to make distinctions in set theory that were previously unknown. However, we also think of the arising inner models as a tool to learn more about extended logics. As it turns out, for many non-equivalent logics the inner model is the same. In particular for many non-elementary logics the inner model is the same as for first order logic. We may think that such logics have some albeit distant similarity to first order logic. On the other hand, some other logics give rise to the inner model $\hod$. We may say that they bear some resemblance to second order logic.

Our main results can be summarized as follows:

\begin{description}

\item[(A)] For the logics $\L(Q_\alpha)$ we obtain just $L$, for any choice of  $\a$. If $0^\#$ exists the same is true of the Magidor-Malitz logics $\L(Q^{\MMa}_\alpha)$.

\item[(B)] If $0^\sharp$ exists the cofinality quantifier logic $\L(Q^{\mbox{\scriptsize cf}}_\omega)$ yields a proper extension ${C^*}$ of $L$. But  ${C^*}\ne HOD$  if there are uncountably many measurable cardinals.

\item[(C)] If there is a proper class of Woodin cardinals, then regular  cardinals $>\aleph_1$ are Mahlo and indiscernible in ${C^*}$, and the theory of 
${C^*}$ is invariant under (set) forcing.

\item[(D)] The Dodd-Jensen Core Model is contained in ${C^*}$. If there is an inner model with a measurable cardinal, then such an inner model is also contained in ${C^*}$.

\item[(E)] If there is a Woodin cardinal and a measurable cardinal above it, then CH is true in the version $C^*(x)$ of $C^*$, obtained by allowing a real parameter $x$, for a cone of reals $x$.

\end{description}

\section{Basic concepts}

We define an analogue of the constructible hierarchy of G\"odel by replacing first order logic in the construction by an arbitrary logic $\L^*$. We think of logics in the sense of Lindstr\"om \cite{MR0244013}, Mostowski \cite{MR0250861}, Barwise \cite{MR0376337}, and the collection \cite{MR819531}. What is essential is that a logic $\L^*$ has two components i.e. $\L^*=(S^*,T^*)$, where $S^*$ is the class of sentences of $\L^*$ and $T^*$ is the truth predicate of $\L^*$. We usually write $\phi\in\L^*$ for $\phi\in S^*$ and $\mm\models\phi$ for $T^*(\mm,\phi)$. We can talk about {\em formulas} with {\em free variables} by introducing new constant symbols and letting the constant symbols play the role of free variables. The classes $S^*$ and $T^*$ may be defined with parameters, as in the case of $\L_{\kappa\lambda}$, where $\kappa$ and $\lambda$ can be treated as parameters.
A logic $\L^*$ is a {\em sublogic} of another logic $\L^+$, $\L^*\le\L^+$, if for every $\phi\in S^*$ there is $\phi^+\in S^+$ such that for all $\mm$: $\mm\models\phi\iff\mm\models\phi^+$. We assume that our logics have first order order logic as sublogic.

\begin{example}\begin{enumerate}
\item {\em First order logic} $\loo$ (or $\fol$) is the logic  $(S^*,T^*)$, where $S^*$ is the set of first order sentences and $T^*$ is the usual truth definition for first order sentences. 

\item {\em Infinitary logic} $\L_{\kappa\lambda}$, where $\kappa$ and $\lambda\le \kappa$ are regular cardinals, is the logic $(S^*,T^*)$, where $S^*$ consists of the sentences built inductively from conjunctions and disjunctions of length $<\kappa$ of sentences of $\L_{\kappa\lambda}$, and homogeneous strings of existential and universal quantifiers of length $<\lambda$ in front of formulas of $\L_{\kappa\lambda}$. The class $T^*$ is defined in the obvious way. We allow also the case that $\kappa$ or $\lambda$ is $\infty$. We use $\L^{\omega}_{\kappa\lambda}$ to denote that class of formulae of $\L_{\kappa\lambda}$ with only finitely many free variables.

\item The logic $\L(Q)$ with a generalized quantifier  $Q$ is  the logic $(S^*,T^*)$, where $S^*$ is obtained by adding the new quantifier $Q$ to first order logic. The exact syntax depends on the type of $Q$, (see our examples below). The class $T^*$ is defined by first fixing the {\em defining  model class} $\K_Q$ of $Q$ and then defining $T^*$ by induction on formulas: $$\mm\models Qx_1,\ldots,x_n\phi(x_1,\ldots,x_n,\vec{b})\iff$$ $$(M,\{(a_1,\ldots,a_n)\in M^n:\mm\models\phi(a_1,\ldots,a_n,\vec{b})\})\in\K_Q.$$
Thought of in this way, the defining model class of the existential quantifier is the class $\K_\exists=\{(M,A): \emptyset\ne A\subseteq M\}$, and the defining model class of the universal quantifier is the class $\K_\forall=\{(M,A):  A= M\}$. Noting that the generalisations  of $\exists$ with defining class $\{(M,A): A\subseteq M, |A|\ge n\}$, where $n$ is fixed, are definable in first order logic, Mostowski \cite{MR0089816} introduced the generalisations  $Q_\alpha$ of $\exists$ with defining class $$\K_{Q_\alpha}=\{(M,A):A\subseteq M, |A|\ge\aleph_\alpha\}.$$ Many other generalized quantifiers are known today in the literature and we will introduce some important ones later. 

\item {\em Second order logic} $\L^2$ 
is  the logic $(S^*,T^*)$, where $S^*$ is obtained from first order logic by adding variables for $n$-ary relations for all $n$ and allowing existential and universal quantification over the new variables.  The class $T^*$ is defined by  the obvious induction. In this inductive definition of $T^*$ the second order variables range over all relations of the domain (and not only e.g. over definable relations). 
\end{enumerate}
\end{example}

We now define the main new concept of this paper:

\begin{definition}\label{defin}
Suppose $\L^*$ is a logic. If $M$ is a set, let $\DEF_{\L^*}(M)$ denote the set of all sets of the form $X=\{a\in M : (M,\in)\models\phi(a,\vec{b})\},$ where $\phi(x,\vec{y})$ is an arbitrary formula of the logic $\L^*$ and $\vec{b}\in M$. We define a hierarchy $(L'_\alpha)$ of {\em sets constructible using} $\L^*$ as follows:

\begin{center}
\begin{tabular}{lcl}
$L'_0$&=&$\emptyset$\\
$L'_{\alpha+1}$&$=$&$\DEF_{{\mathcal L}^*}(L'_\alpha)$\\
$L'_\nu$&$=$&$\bigcup_{\alpha<\nu}L'_\alpha\mbox{ for limit $\nu$}$\\
\end{tabular} 
\end{center}

\noindent We use $C(\L^*)$ to denote the class $\bigcup_\alpha L'_\alpha$.
\end{definition}

Thus a typical set in $L'_{\alpha+1}$ has the form 
\begin{equation}
\label{typical}X=\{a\in L'_\alpha:(L'_\alpha,\in)\models\phi(a,\vec{b})\}
\end{equation}
where $\phi(x,\vec{y})$ is a formula of $\L^*$ and $\vec{b}\in L'_\alpha$.
It is important to note that $\phi(x,\vec{y})$ is a formula of $\L^*$ {\em in the sense of $V$}, not in the sense of $C(\L^*)$, i.e. we assume $S^*(\phi(x,\vec{y}))$ is {\em true} rather than being true in $(L'_\alpha,\in)$. In extensions of first order logic of the form $\L(Q)$ this is not a problem because being a formula is absolute to high degree. For example, in a countable vocabulary we have G\"odel-numbering for the set of formulas of $\L(Q)$ which renders the set of G\"odel-numbers of formulas primitive recursive. On the other hand, the set of formulas of $\L_{\omega_1\omega}$ is highly non-absolute, because an infinite conjunction may be uncountable in $L'_\alpha$ but countable in $V$. Also, note that $(L'_\alpha,\in)\models\phi(a,\vec{b})$ refers to $T^*$ in the sense of $V$, not in the sense of $C(\L^*)$. This is a serious point. For example, if $\phi(a,\vec{b})$ compares cardinalities or cofinalities of $a,\vec{b}$ to each other, the witnessing mappings do not have to be in $L'_\alpha$.  
   
By definition, $C(\loo)=L$. Myhill-Scott \cite{MR0281603} showed that $C(\L^2)=\hod$ (See Theorem~\ref{msc} below). Chang \cite{MR0280357} considered $C(\L_{\omega_1\omega_1})$ and pointed out that this is the smallest transitive model of {\zfc} containing all ordinals and closed under countable sequences.  Kunen \cite{MR0337603} showed that $C(\L_{\omega_1\omega_1})$ fails to satisfy the Axiom of Choice, if we assume the existence of uncountably many measurable cardinals (see Theorem~\ref{knn} below). Sureson \cite{MR836429, MR991631} investigated a Covering Lemma for $C(\L_{\omega_1\omega_1})$.

\begin{proposition}
For any $\L^*$ the class $C(\L^*)$ is a transitive model of ZF containing all the ordinals. 
\end{proposition}

\begin{proof}
As in the usual proof of ZF in $L$. Let us prove the Comprehension Schema as an example. Suppose $A,\vec b$ are in $C(\L^*)$, $\phi(x,\vec y)$ is a  first order formula of set theory and 
$$X=\{a\in A:C(\L^*)\models\phi(a,\vec b)\}.$$ Let $\alpha$ be an ordinal such that
$A\in L'_\alpha$ and  $\phi(x,y)$ is absolute for $L'_\alpha,C(\L^*)$
(see e.g. \cite[IV.7.5]{MR2905394}). Now
$$X=\{a\in L'_\alpha : L'_\alpha\models a\in A\wedge\phi(a,\vec b)\}.$$
Hence $X\in C(\L^*)$. \end{proof}

We cannot continue and follow the usual proof of AC in $L$, because the {\em syntax} of $\L^*$ may introduce sets into $C(\L^*)$ without introducing a well-ordering for them (See Theorem~\ref{R}). Also, formulas of $\L^*$, such as $\bigvee_nx=y_n$ if in $\looo$, may  contain infinitely many free variables and that could make $C(\L^*)$ closed under $\omega$-sequences, rendering it vulnerable to the failure of AC. To overcome this difficulty, we introduce the following concept, limiting ourselves to logics in which every  formula has only finitely many free variables:

\def\sat{\mbox{Sat}}
\newcommand{\topst}[1]{\ensuremath{\ulcorner{}#1\urcorner}}

\begin{definition}\label{4r67}
 A logic $\L^*$ is {\em  adequate to truth in itself\footnote{This is a special case of a concept with the same name in \cite{MR0491139}.}} if for all finite vocabularies $K$  there is function $\phi\mapsto\topst{\phi}$ from all formulas $\phi(x_1,\ldots, x_n)\in \L^*$ in the vocabulary $K$ into $\omega$, and  a formula $\sat_{\L^*}(x,y,z)$ in $\L^*$ such that:
 
\begin{enumerate}
\item The function $\phi\mapsto\topst{\phi}$ is one to one and has a recursive range.
\item For all admissible sets\footnote{I.e. transitive models of the Kripke-Platek axioms $KP$ of set theory. The only reason why we need admissibility is that admissible sets are closed under inductive definitions of the simple kind that are used in the syntax and semantics of many logics. For more on admissibility we refer to \cite{MR0424560}.} $M$, formulas $\phi$ of $\L^*$ in the vocabulary $K$,  structures $\mn\in M$ in the vocabulary $K$, and $a_1,\ldots,a_n\in N$ the following conditions are equivalent:
\begin{enumerate}
 
\item $ M\models\sat_{\L^*}(\mn,\topst{\phi},\langle a_1,\ldots, a_n\rangle)$
\item $\mn\models\phi(a_1,\ldots, a_n).$
\end{enumerate}
We may admit ordinal parameters in this definition.
\end{enumerate}
\end{definition}

Most logics that one encounters in textbooks and research articles of logic are adequate to truth in themselves. To find counter examples one has to consider e.g. logics with infinitely many generalized quantifiers.  

\begin{example}
First order logic $\loo$ and 
the logic $\L(Q_\alpha)$ are adequate to truth in themselves.
Also second order logic is adequate to truth in itself in the slightly weaker sense that $M$ has to be of size $\ge 2^{|N|}$ because we also have second order variables. Infinitary logics are for obvious reasons (cannot use natural numbers for G\"odel-numbering) not adequate to truth in themselves, but there is a more general notion which applies to them (see \cite{MR0491139,MR819548}). In infinitary logic what accounts as a formula depends on set theory. For example, in the case of $\L_{\omega_1\omega}$ the formulas essentially code in their syntax all reals.

\end{example}

The following proposition is instrumental in showing that $C(\L^*)$, for certain $\L^*$, satisfies the Axiom of Choice:

\begin{proposition}\label{wpoa}
 If $\L^*$ is adequate to truth in itself, there are formulas $\Phi_{\L^*}(x)$ and $\Psi_{\L^*}(x,y)$ of $\L^*$ in the vocabulary $\{\in\}$ such that if $M$ is an admissible set and $\a=M\cap \On$, then:
 
\begin{enumerate}
\item $\{a\in M : (M,\in)\models\Phi_{\L^*}(a)\}=L'_\a\cap M.$
\item $\{(a,b)\in M\times M : (M,\in)\models\Psi_{\L^*}(a,b)\}$ is a well-order $<'_\a$ the field of which is $L'_\a\cap M$. 
\end{enumerate}
\end{proposition}

It is important to note that the formulas $\Phi_{\L^*}(x)$ and $\Psi_{\L^*}(x,y)$ are in the extended logic $\L^*$, not necessarily in first order logic.

Recall that we have defined the logic $\L^*$ as a pair $(S^*,T^*)$. We can use the set-theoretical predicates $S^*$ and $T^*$ to write  $``(M,\in)\models\Phi_{\L^*}(a)"$ and $``(M,\in)\models\Psi_{\L^*}(x,y)"$ of Proposition~\ref{wpoa} as formulas
$\tilde{\Phi}_{\L^*}(M,x)$ and $\tilde{\Psi}_{\L^*}(M,x,y)$ of the first order language of set theory, such that for all $M$ with $\a=M\cap\On$ and $a,b\in M$:
 
\begin{enumerate}
\item $\tilde{\Phi}_{\L^*}(M,a)\leftrightarrow[(M,\in)\models\Phi_{\L^*}(a)]\leftrightarrow a\in L'_\a$.
\item $\tilde{\Psi}_{\L^*}(M,a,b)\leftrightarrow[(M,\in)\models\Psi_{\L^*}(a,b)]\leftrightarrow a<'_\a b$.
\end{enumerate}

\begin{proposition}\label{axc}
 If $\L^*$ is adequate to truth in itself, then $C(\L^*)$ satisfies  the Axiom of Choice.
\end{proposition}

\begin{proof} Let us fix $\a$ and show that there is a well-order of $L'_\a$ in $C(\L^*)$. Let $\kappa=|\a|^+$. Then $\Psi_{\L^*}(x,y)$ defines on $L'_{\kappa}$ a well-order 
$<'_{\kappa}$ of $L'_{\kappa}$. The relation $<'_\kappa$ is in  $L'_{\kappa+1}\subseteq C(\L^*)$ by the definition of $C(\L^*)$. \end{proof}

There need not be a first order definable well-order of the class $C(\L^*)$ 
(see the proof of Theorem~\ref{wrep} for an example) although there always is in $V$ a definable relation which well-orders $C(\L^*)$. Of course, in this case $V\ne C(\L^*)$. Proposition~\ref{axc} holds also for second order logic, even though it is only adequate to truth in itself in a slight weaker sense.

Note that trivially  $$\L^*\le \L^+\mbox{ implies } C(\L^*)\subseteq C(\L^+).$$
Thus varying the logic $\L^*$ we get a whole hierarchy of inner models $C(\L^*)$.
Many questions can be asked about these inner models. For example we can ask: (1) can all the known inner models be obtained in this way, (2) under which conditions do these inner models satisfy GCH, (3) do inner models obtained in this way have other characterisations (such as $L$, $\hod$ and $C(\L_{\omega_1\omega_1})$ have), etc.

\begin{definition}
A set $a$ is {\em ordinal definable} if there is a formula $\phi(x,y_1,\ldots,y_n)$ and ordinals $\alpha_1,\ldots,\alpha_n$ such that \begin{equation}\label{od}\forall x(x\in a\iff\phi(x,\alpha_1,\ldots,\alpha_n)).\end{equation} A set $a$ is {\em hereditarily ordinal definable} if $a$ itself and also every element of $\TC(a)$ is ordinal definable. 

\end{definition}

When we look at the construction of $C(\L^*)$ we can observe that sets in $C(\L^*)$ are always hereditarily ordinal definable when the formulas of $\L^*$ are finite (more generally, the formulas may be hereditarily ordinal definable):

\begin{proposition}
If $\L^*$ is any logic such that the formulas $S^*$ and $T^*$ do not contain parameters (except hereditarily ordinal definable ones) and in addition every formula of $\L^*$ (i.e. element of the class $S^*$) is a finite string of symbols (or more generally hereditarily ordinal definable, with only finitely many free variables), then every set in $C(\L^*)$ is hereditarily ordinal definable.
\end{proposition}

\begin{proof} 
Recall the  construction of the successor stage of $C(\L^*)$: 
$X\in L'_{\alpha+1}$ 
if and only if for some $\phi(x,\vec{y})\in \L^*$ and some $\vec{b}\in L'_\alpha$
  $$X=\{x\in L'_\alpha : (L'_\alpha,\in)\models\phi(x,\vec{b})\}.$$ Now we can note that 
  $$X=\{x\in L'_\alpha : T^*((L'_\alpha,\in),\phi(x,\vec{b}))\}.$$ %
Thus if $L'_\alpha$ is ordinal definable, then so is $X$. Moreover,
$$\forall z(z\in L'_{\alpha+1}\iff\exists \phi(x,\vec{u})( S^*(\phi(x,\vec{u}))\wedge$$
$$\forall y(y\in z\iff y\in L'_\alpha  \wedge  T^*((L'_\alpha,\in),\phi(x,\vec{u}))),$$ or in short
$$\forall z(z\in L'_{\alpha+1}\iff\psi(z,L'_\alpha)),$$ where $\psi(z,w)$ is a first order  formula
in the language of set theory. When we compare this with (\ref{od}) we see that if $L'_\alpha$ is ordinal definable and if the (first order) set-theoretical formulas $S^*$ and $T^*$ have no parameters, then also $L'_{\alpha+1}$ is ordinal definable.  It  follows that the class $\la L'_\a : a\in\On\ra$ is ordinal definable, whence $\la L'_\a : \a<\nu\ra$, and thereby also $L'_\nu$, is in $\hod$  for all limit $\nu$.
%
\end{proof}

 Thus, unless the formulas of the logic $\L^*$ are syntactically complex (as happens in the case of infinitary logics like $\looo$ and $\L_{\omega_1\omega_1}$, where a formula can code an arbitrary real), the hereditarily ordinal definable sets form a firm ceiling for  the inner models  $C(\L^*)$.

%
%
%
%
%
%
%
%

\begin{theorem}\label{V}
$C(\lio)=V.$
\end{theorem}

\begin{proof} Let $(L'_\alpha)_\alpha$ be the hierarchy behind $C(\lio)$, as in Definition~\ref{defin}. We  show $V_\alpha\subseteq C(\lio)$ by induction on $\alpha$.
For any set $a$ let the formulas $\theta_a(x)$ of set theory be defined by the following transfinite recursion: $$\theta_a(x)=\bigwedge_{b\in a}\exists y(yEx\wedge\theta_b(y))\wedge
\forall y(yEx\to\bigvee_{b\in a}\theta_b(y)).$$ Note that in any transitive set $M$ containing $a$:$$(M,\in)\models\forall x(\theta_a(x)\iff x=a).$$ Let us assume $V_\alpha\subseteq C(\lio)$, or more exactly, $V_\alpha\in L'_\beta$. Let $X\subseteq V_\alpha$. Then $$X=\{a\in L'_\beta : L'_\beta\models a\in V_\alpha\wedge\bigvee_{b\in X}\theta_b(a)\}\in L'_{\beta+1}.$$   \end{proof}

Note that the proof actually shows $C(\L^\omega_{\infty\omega})=V$.

\begin{theorem}\label{R}
$C(\L^\omega_{\omega_1\omega})=L(\oR).$
\end{theorem}

\begin{proof} Let $(L'_\alpha)_\alpha$ be the hierarchy behind $C(\looo^\omega)$. We first show $L(\oR)\subseteq C(\looo^\omega)$.
Since $C(\looo^\omega)$ is clearly a transitive model of ZF it suffices to show that $\oR\subseteq C(\looo^\omega)$. Let $X\subseteq\omega$. Let $\phi_n(x)$ be a formula of set theory which defines the natural number $n$ in the obvious way. Then $$X=\{a\in L'_\alpha : L'_\alpha\models a\in\omega\wedge\bigvee_{n\in X}\phi_n(a)\}\in L'_{\alpha+1}.$$ Next we show $C(\looo^\omega)\subseteq L(\oR)$. We prove by induction on $\alpha$ that $L'_\alpha\subseteq L(\oR)$. Suppose this has been proved for $\alpha$ and $L'_\alpha\in L_\beta(\oR)$. Suppose $X\in L'_{\alpha+1}$. This means that there is a formula $\phi(x,\vec{y})$ of $\looo^\omega$ and a finite sequence $\vec{b}\in L'_{\alpha}$ such that $$X=\{a\in L'_\alpha : L'_\alpha\models\phi(a,\vec{b})\}.$$ 
It is possible (see e.g. \cite[page 83]{MR0424560}) to write  a first order formula $\Phi$ of set theory such that 
$$X=\{a\in L_\beta(\oR) : L_\beta(\oR)\models\Phi(a,L'_{\alpha},\phi,\vec{b})\}.$$
Since there is a canonical coding of formulas of $\looo^\omega$ by reals we can consider $\phi$ as a real parameter. Thus $X\in L_{\beta+1}(\oR)$. \end{proof}

\begin{theorem}
 $C(\L_{\omega_1\omega})=C(\L_{\omega_1\omega_1})\hspace{2mm}(=\mbox{Chang model}).$
\end{theorem}

\begin{proof}
 The model $C(\L_{\omega_1\omega})$ is closed under countable sequences, for if $a_n\in C(\L_{\omega_1\omega})$ for $n<\omega$, then the $\L_{\omega_1\omega}$-formula
 $$\forall y(y\in x\leftrightarrow\bigvee_n y=\langle n,a_n\rangle).$$
defines the sequence $\la a_n:n<\omega\ra$. Since the Chang model is the smallest transitive model of $ZF$ closed under countable sequences, the claim follows.\end{proof}

We already known that several familiar inner models ($L$ itself, $L(\oR)$, Chang model) can be recovered in the form $C(\L^*)$. We can also recover the inner model $L^\mu$ of one measurable cardinal as a model of the form $C(\L^*)$ in the following somewhat artificial way:
 
\begin{definition} Suppose $U$ is a normal ultrafilter on $\kappa$. We define a generalised quantifier $Q^\kappa_U$ as follows:
 $$\mm\models Q_\kappa^Uwxyv\theta(w)\phi(x,y)\psi(v)\iff$$
 $$\exists\pi:(S,R)\cong(\kappa,<)\wedge\pi''A\in U,$$ where
\begin{eqnarray*}
S&=&\{a\in M : \mm\models\theta(a)\}\\
R&=&\{(a,b)\in M^2 : \mm\models\phi(a,b)\}\\
A&=&\{a\in M : \mm\models\psi(a)\} 
\end{eqnarray*}

\end{definition}

\begin{theorem}
 $C(Q_{\kappa}^U)=L^U$.
\end{theorem}

\begin{proof}Let $(L'_\a)$ be the hierarchy that defines $C(Q_{\kappa}^U)$.
 We prove for all $a$: $L'_\a=L^U_\a$. We use induction on $\a$.
Suppose the claim is true up to $\a$.
   Suppose $X\in L'_{\a+1}$, e.g.
$$X=\{a\in L'_\a : (L'_a,\in)\models\phi(a,\vec{b})\},$$
where $\phi(x,\vec{y})\in FO(Q_{\kappa}^U)$ and $\vec{b}\in L'_\a$. We show $X\in L_\a^U$. To prove this  we use induction on $\phi(x,\vec{y})$. Suppose  
$$X=\{a\in L'_\a : (L'_\a,\in)\models Q_{\kappa}^Uwxyv\theta(z,a,\vec{b})\phi(x,y,a,\vec{b})\psi(v,a,\vec{b})\}$$ and the claim has been proved for  $\theta$, $\phi$ and $\psi$. Let 
$$Y_a=\{c\in L'_\a : (L'_\a,\in)\models \theta(c,a,\vec{b})\},$$
$$R_a=\{(c,d)\in L'^2_\a : (L'_\a,\in)\models \phi(c,d,a,\vec{b})\},$$
and 
$$S_a=\{c\in L'_\a : (L'_\a,\in)\models \psi(c,a,\vec{b})\}.$$
Thus
$$X=\{a\in L'_\a : \exists \pi:(Y_a,R_a)\cong (\kappa,<)\wedge \pi''S_a\in U\}.$$
But now 
$$X=\{a\in L_\a^U : \exists \pi:(Y_a,R_a)\cong (\kappa,<)\wedge \pi''S_a\in U\cap L^U\},$$ so $X\in L^U$.
\medskip

\noindent
{\bf Claim 2:\ } For all $a$: $L^U_\a\in C(Q_{\kappa}^U)$. We use induction on $\a$. It suffices to prove for all $\a$: $U\cap L^U_\a\in C(Q_{\kappa}^U)$. Suppose the claim is true up to $\a$. We show $U\cap L^U_{\a+1}\in C(Q_{\kappa}^U)$. Now
$$U\cap L^U_{\a+1}=U\cap \DEF(L^U_\a,\in,U\cap L^U_\a)$$

$$=\{X\subseteq L^U_\a : X\in U\wedge X\in \DEF(L^U_\a,\in,U\cap L^U_\a)\}$$

\end{proof}

\section{Absolute logics}

The concept of an absolute logic attempts to capture the first-order content of $\loo$. Is it possible that  logics that are ``first order" in the  way  $\loo$ is turn out to be substitutable with $\loo$ in the definition of the constructible hierarchy?

 Barwise writes in \cite[pp. 311-312]{MR0337483}: \begin{quote}``Imagine a logician $\mathpzc{k}$  using $T$ as his metatheory for defining the basic notions of a particular logic $\L^*$. When is it reasonable for us, as outsiders looking on, to call $\L^*$ a ``first order" logic? If the words ``first order" have any intuitive content it is that the truth or falsity of $\mm\models^*\phi$
should depend only on $\phi$ and $\mm$, not on what subsets of $M$ may or may not exist in $\mathpzc{k}$'s model of his set theory $T$. In other words, the relation  $\models^*$ should be absolute for models of $T$. What about the predicate
$\phi\in \L^*$
of $\phi$? To keep from ruling out $L_{\omega_1\omega}$ (the predicate $\phi\in L_{\omega_1\omega}$ is not absolute since the notion of countable is not absolute) we demand only that the notion of $\L^*$-sentence be persistent for models of $T$: i.e. that if $\phi\in \L^*$ holds in $\mathpzc{k}$'s model of $T$ then it should hold in any end extension of it."
 \end{quote}
Using absoluteness as a guideline, Barwise \cite{MR0337483} introduced the concept of an absolute logic:
\begin{definition}\label{wqe}
Suppose $A$ is any class and $T$ is any theory in the language of set theory. A logic $\L^*$ is {\em $T$-absolute} if there are a $\Sigma_1$-predicate $S_1(x)$, a $\Sigma_1$-predicate $T_1(x,y)$, and a $\Pi_1$-predicate $T'_1(x,y)$ such that $\phi\in \L^*\iff S_1(\phi)$, $M\models\phi\iff T_1(M,\phi)$ and $T\vdash\forall x\forall y(S_1(x)\to (T_1(x,y)\leftrightarrow T'_1(x,y)))$. If parameters from a class $A$ are allowed, we say that $\L^*$ is {\em absolute with parameters from $A$}.
\end{definition}

Note that the stronger $T$ is, the weaker the notion of $T$-absoluteness is. 
Barwise \cite{MR0337483} calls KP\footnote{Kripke-Platek set theory.}-absolute logics {\em strictly} absolute.

As Theorems \ref{V} and $\ref{R}$ demonstrate, absolute logics (such as $\looo^\omega$) may be very strong from the point of view of the inner model construction. However, this is  so only because of  the potentially complex syntax of the absolute logics, as is the case with $\looo$. Accordingly we introduce the following notion: 

\begin{definition}
 An absolute logic $\L^*$ has  {\em $T$-absolute syntax} if its sentences are (coded as) natural numbers and there is a $\Pi_1$-predicate $S'_1(x)$ such that $T\vdash\forall x(S_1(x)\leftrightarrow S'_1(x))$.  We may allow parameters, as in Definition \ref{wqe}.
\end{definition}

In other words, to say that a logic $\L^*$ has ``absolute syntax" means that the class of $\L^*$-formulas has a $\Delta_1^T$-definition. Obviously, $\looo$ does not satisfy this condition.  On the other hand, many absolute logics, such as $\loo$, $\L(Q_0)$, weak second order logic, $\L_{\mbox{\tiny HYP}}$, etc have absolute syntax.

The original definition of absolute logics does not allow parameters. Still there are many logics that are absolute apart from dependence on a parameter. In our context it turns out that we can and should allow parameters.

The {\bf cardinality quantifier} $Q_\alpha$ is defined as follows: $$\mm\models Q_\alpha x\phi(x,\vec{b})\iff
|\{a\in M:\mm\models\phi(a,\vec{b})\}|\ge\aleph_\alpha.$$ A slightly stronger quantifier is
$$\mm\models Q^E_\alpha x,y\phi(x,y,\vec{c})\iff
\{(a,b)\in M^2:\mm\models\phi(a,b,\vec{c})\}\mbox{ is an}$$
$$\mbox{\hspace{2in} equivalence relation with $\ge \aleph_\alpha$ classes.}$$
 
\begin{example}
\begin{enumerate}
\item $\lio$ is KP-absolute \cite{MR0337483}.
\item $\L(Q_\alpha)$ is {\zfc}-absolute with $\omega_\alpha$ as parameter.
\item $\L(Q^E_\alpha)$ is {\zfc}-absolute with $\omega_\alpha$ as parameter.
\end{enumerate}
\end{example}

\begin{theorem}\label{abs}
Suppose $\L^*$ is {\zfc}+V=L-absolute with parameters from $L$, and 
the syntax of $\L^*$ is ({\zfc}+V=L)-absolute with parameters from  $L$. Then $C(\L^*)=L$.
\end{theorem}

\begin{proof} 
 We use induction on $\alpha$ to prove that $L'_\alpha\subseteq L$. We suppose $L'_\alpha\subseteq L$ and  that ${\zfc}_n$ is a finite part of ${\zfc}$ so that $\L^*$ is ${\zfc}_n+V=L$-absolute. Then $L'_\alpha\in L_\gamma$ for some  $\gamma$ such that $L_\gamma\models {\zfc}_n$. We show that $L'_{\alpha+1}\subseteq L_{\gamma+1}$. Suppose $X\in L'_{\alpha+1}$. Then $X$ is of the form $$X=\{a\in L'_{\alpha}: (L'_{\alpha},\in)\models\phi(a,\vec{b})\},$$ where $\phi(x,\vec{y})\in \L^*$ and $\vec{b}\in L'_{\alpha}$. W.l.o.g., $\phi(x,\vec{y})\in L_\gamma$. By the definition of absoluteness,
$$X=\{a\in L_{\gamma}: (L_\gamma,\in)\models a\in L'_\alpha\wedge S_1(\phi(x,\vec{y}))\wedge T_1(L'_{\alpha},\phi(a,\vec{b}))\}.$$
Hence $X\in L_{\gamma+1}$. This also shows  that $\la L'_\a : \a<\nu\ra\in L$, and thereby $L'_\nu\in L$, for limit ordinals $\nu$.
\end{proof}
A consequence of the Theorem \ref{abs} is the following:
\medskip

\noindent{\bf Conclusion:} The constructible hierarchy $L$ is unaffected if first order logic is exchanged in the construction of $L$ for any of the following, simultaneously or separately:
\begin{itemize}
\item Recursive infinite conjunctions $\bigwedge_{n=0}^\infty\phi_n$ and disjunctions $\bigvee_{n=0}^\infty\phi_n$.

\item Cardinality quantifiers $Q_\alpha$, $\alpha\in On$.

\item Equivalence quantifiers $Q^E_\alpha$, $\alpha\in On$.

\item Well-ordering quantifier\footnote{This quantifier is absolute because the well-foundedness of a linear order $<$  is equivalent to the existence of a function from the tree of strictly $<$-decreasing sequences into the ordinals such that a strictly longer sequence is always mapped to a strictly  smaller ordinal.} $$\mm\models W x,y\phi(x,y)\iff$$
$$\{(a,b)\in M^2:\mm\models\phi(a,b)\}\mbox{ is a well-ordering}.$$

\item Recursive game quantifiers $$\forall x_0\exists y_0\forall x_1\exists y_1\ldots\bigwedge_{n=0}^\infty\phi_n(x_0,y_0,\ldots,x_n,y_n),$$
$$\forall x_0\exists y_0\forall x_1\exists y_1\ldots\bigvee_{n=0}^\infty\phi_n(x_0,y_0,\ldots,x_n,y_n).$$

\item Magidor-Malitz quantifiers\footnote{\label{dqd}This quantifier is absolute because the existence of an infinite set $X$ as above is equivalent to the non-well-foundedness of the tree of strictly $\subset$-increasing sequences $(s_0,\ldots,s_m)$ of finite subsets of the model with the property that $\mm\models\phi(a_1,\ldots,a_n)$ holds for all $a_1,\ldots,a_n\in s_m$.} at $\aleph_0$ $$\mm\models Q^{\MMa,n}_0 x_1,\ldots,x_n\phi(x_1,\ldots,x_n)\iff$$
$$\exists X\subseteq M(|X|\ge\aleph_0\wedge \forall a_1,\ldots,a_n\in X:\mm\models\phi(a_1,\ldots,a_n)).$$

\end{itemize}
Thus G\"odel's $L=C(\loo)$ exhibits some robustness with respect  to the choice of the logic.

\section{The Magidor-Malitz quantifier}
The Magidor-Malitz quantifier at $\aleph_1$ \cite{MR0453484} extends $Q_1$ by allowing us to say that there is an uncountable set such that, not only every  {\em element} of the set satisfies a given formula $\phi(x)$, but even any {\em pair of elements} from the set  satisfy a given formula $\psi(x,y)$. Much more is expressible with the Magidor-Malitz quantifier than with $Q_1$, e.g. the existence of a long branch or of a long antichain in a tree, but this quantifier is still axiomatizable  if one assumes $\diamondsuit$. On the other hand, the price we pay for the increased expressive power is that it is consistent, relative to the consistency of ZF, that  Magidor-Malitz logic is very badly incompact \cite{MR1197203}. We show that while it is consistent, relative to the consistency of ZF, that the Magidor-Malitz logic generates an inner model different from $L$, if we assume $0^\sharp$, the inner model collapses to $L$.  
This is a bit surprising, because the existence of $0^\sharp$ implies that $L$ is very ``slim", in the sense that it is not something that an a priori bigger inner model would collapse to. The key to this riddle is that under $0^\sharp$ the Magidor-Malitz logic itself loses its ``sharpness" and becomes in a sense absolute between $V$ and $L$.

\begin{definition}The
Magidor-Malitz quantifier in dimension $n$ is the following: $$\mm\models Q^{\MMa,n}_\alpha x_1,\ldots,x_n\phi(x_1,\ldots,x_n)\iff$$
$$\exists X\subseteq M(|X|\ge\aleph_\alpha\wedge \forall a_1,\ldots,a_n\in X:\mm\models\phi(a_1,\ldots,a_n)).$$
\end{definition}
The original Magidor-Malitz quantifier had dimension $2$ and $\alpha=1$:
$$\mm\models Q^{\MMa}_1 x_1,x_2\phi(x_1,x_2)\iff$$
$$\exists X\subseteq M(|X|\ge\aleph_1\wedge \forall a_1,a_2\in X:\mm\models\phi(a_1,a_2)).$$
%

The logics $\L(Q^{\MMa,<\omega}_\kappa)$ and $\L(Q^{\MMa,n}_\kappa)$ are adequate to truth in themselves (recall Definition~\ref{4r67}), with $\kappa$ as a parameter. 

Note that putting $n=1$ gives us $Q_1$: 
$$\mm\models Q_1 x\phi(x)\iff$$
$$\exists X\subseteq M(|X|\ge\aleph_1\wedge \forall a\in X:\mm\models\phi(a)).$$
We have already noted in Footnote~\ref{dqd} that for $\alpha=0$ this quantifier is absolute.

\begin{theorem}
If $0^\sharp$ exists, then $C(Q^{\MMa,<\omega}_\alpha)=L$.
\end{theorem}

\begin{proof}
We treat only the case $n=2, \alpha=1$. The general case is treated similarly, using induction on $n$. The proof hinges on the following lemma:

\begin{lemma} Suppose $0^\sharp$ exists and
 $A\in L$, $A\subseteq [\eta]^2$. If there is an uncountable $B$ such that $[B]^2\subseteq A$, then there is  such a set $B$ in $L$.
\end{lemma}

\begin{proof}
Let us first see how the Lemma helps us to prove the theorem. We will use induction on $\alpha$ to prove that $L'_\alpha\subseteq L$. We suppose $L'_\alpha\subseteq L$, and hence $L'_\alpha\in L_\gamma$ for some canonical indiscernible $\gamma$. We show that $L'_{\alpha+1}\subseteq L_{\gamma+1}$. Suppose $X\in L'_{\alpha+1}$. Then $X$ is of the form $$X=\{a\in L'_{\alpha}: (L'_{\alpha},\in)\models\phi(a,\vec{b})\},$$ where $\phi(x,\vec{y})\in \L(Q^{\MMa}_1)$ and $\vec{b}\in L'_{\alpha}$. For simplicity we suppress the mention of $\vec{b}$. Since we can use induction on $\phi$, the only interesting case is
$$X=\{a\in L'_{\alpha}: \exists Y(|Y|\ge\aleph_1\wedge\forall 
x,y\in Y:(L'_{\alpha},\in)\models \psi(x,y,a))\},$$ where we already have
for each $a\in L'_\alpha$
$$A=\{\{c,d\}\in [L'_{\alpha}]^2: (L'_{\alpha},\in)\models \psi(c,d,a)\}\in L.$$ Now the Lemma implies
$$X=\{a\in L'_{\alpha}: \exists Y\in L(|Y|\ge\aleph_1\wedge\forall 
x,y\in Y:(L'_{\alpha},\in)\models \psi(x,y,a))\}.$$ 
Since $L_\gamma\prec L$, we have
$$X=\{a\in L'_{\alpha}: \exists Y\in L_\gamma(|Y|\ge\aleph_1\wedge\forall 
x,y\in Y:(L'_{\alpha},\in)\models \psi(x,y,a))\}.$$ 
Finally, 
$$X=\{a\in L_{\gamma}: (L_\gamma,\in)\models a\in L'_\alpha\wedge$$
$$\exists z(``\exists f:(\aleph_1)^V\stackrel{{\tiny 1-1}}{\to} z"\wedge\forall 
x,y\in z \psi(x,y,a)^{(L'_{\alpha},\in)})\}.$$
\end{proof}
\def\g{\gamma}\def\t{\tau}\def\a{\alpha}\def\b{\beta}\def\d{\delta}\def\U{\mathcal{U}}
Now we prove the Lemma. W.l.o.g. the set $B$ of the lemma satisfies $|B|=\aleph_1$, say $B=\{\delta_i:i<\omega_1\}$ in increasing order. Let $I$ be the canonical closed unbounded class of indiscernibles for $L$. Let $\delta_i=\t_i(\a^i_0,\ldots,\a^i_{k_i})$, where $\a^i_0,\ldots,\a^i_k\in I$. W.l.o.g., $\t_i$ is a fixed term $\t$. Thus also $k_i$ is a fixed number $k$. By the $\Delta$-lemma, by thinning $I$ if necessary, we may assume that the finite sets $\{\a^i_0,\ldots,\a^i_k\}, i<\omega_1$, form a $\Delta$-system with a root $\{\a_0,\ldots,\a_n\}$ and leaves $\{\b^i_0,\ldots,\b^i_k\}$, $i<\omega_1$.  
W.l.o.g. the mapping $i\mapsto\b^i_0$ is strictly increasing in $i$. Let $\g_0=\sup\{\b^i_0:i<\omega_1\}$.
W.l.o.g., the mapping $i\mapsto\b^i_1$ is also strictly increasing in $i$. Let $\g'_1=\sup\{\b^i_1:i<\omega_1\}$. It may happen that $\g_1=\g_0$. Then we continue to $\b^i_2$, $\b^i_3$, etc until we get $\g'_{k_0}=\sup\{\b^i_{k_0}:i<\omega_1\}>\g_0$. Then we let $\g_1=\g'_{k_0}$. We continue in this way until we have $\g_0<...<\g_{k_s-1}$, all limit points of $I$.

Recall that whenever $\g$ is a limit point of the set $I$ there is a natural $L$-ultrafilter $\U_\g\subseteq L$ on $\g$, namely $A\in\U_\g\iff\exists\delta<\g((I\setminus\delta)\cap\g\subseteq A)$.
Recall also the following property of the $L$-ultrafilters $\U_\g$:
\begin{itemize}
\item {\bf Rowbottom Property}: Suppose $\g_1<...<\g_n$ are limits of indiscernibles and 
$\U_{\g_1}$,\ldots,$\U_{\g_n}$ are the corresponding $L$-ultrafilters. Suppose $C\subseteq [\g_1]^{n_1}\times...\times [\g_l]^{n_l}$, where $C\in L$. Then there are $B_1\in\U_{\g_1},\ldots,B_l\in\U_{\g_l}$ such that 
\begin{equation}
\label{rp}
[B_1]^{n_1}\times...\times [B_l]^{n_l}\subseteq C\mbox{ or }[B_1]^{n_1}\times...\times [B_l]^{n_l}\cap C=\emptyset.\end{equation}
\end{itemize}

We apply this to the ordinals $\g_1,\ldots,\g_{k_s-1}$ and to a set $C$ of sequences 
\begin{equation}\label{seq}
(\zeta^0_0,\ldots,\zeta^0_{k_0-1},\eta^0_{k_{s-1}},\ldots,\eta^0_{k_0-1},\ldots,
\zeta^s_{k_{s-1}},\ldots,\zeta^s_{k_s-1},\eta^s_{k_{s-1}},\ldots,\eta^s_{k_s-1})
\end{equation} such that 
\begin{equation}\label{C}
\begin{array}{l}
\{\tau(\alpha_0,\ldots,\alpha_n,\zeta^0_0,\ldots,\zeta^0_{k_0-1},\ldots,
\zeta^s_{k_{s-1}},\ldots,\zeta^s_{k_s-1}),\\
\hspace{2mm}\tau(\alpha_0,\ldots,\alpha_n,\eta^0_0,\ldots,\eta^0_{k_0-1},\ldots,
\eta^s_{k_{s-1}},\ldots,\eta^s_{k_s-1})\}\in A
\end{array}\end{equation} Since $A\in L$, also $C\in L$. Note that 
$$C\subseteq [\g_1]^{2k_0}\times...\times [\g_s]^{2k_s}$$
By the Rowbottom Property there are $B_0\in\U_{\g_0},\ldots,B_s\in\U_{\g_s}$ such that 
\begin{equation}
\label{rp1}
[B_1]^{2k_0}\times...\times [B_s]^{2k_s}\subseteq C\mbox{ or }[B_1]^{2k_0}\times...\times [B_s]^{2k_s}\cap C=\emptyset.\end{equation}
\medskip

\noindent{\bf Claim: }\ $[B_1]^{2k_0}\times...\times [B_s]^{2k_s}\subseteq C$. 
\medskip

To prove the claim suppose $[B_1]^{2k_0}\times...\times [B_s]^{2k_s}\cap C=\emptyset$.
Since $B_j\in\U_{\g_j}$, there is $\xi_j<\g_j$ such that $(I\setminus\xi_j)\cap\g_j\subseteq B_j$. We can now find $i_1,i_2<\omega_1$ such that in the sequence 
$$\b^{i_l}_0,\ldots,\b^{i_l}_{k_0-1},\ldots,\b^{i_l}_{k_{s-1}},\ldots,\b^{i_l}_{k_s-1}, l\in\{1,2\},$$
where  $$\b^{i_l}_0,\ldots,\b^{i_l}_{k_0-1}<\g_0\mbox{ and }\b^{i_l}_{k_{j-1}},\ldots,\b^{i_l}_{k_j-1}<\g_j \mbox{ for all $j$},$$
we actually have 
$$\xi_0<\b^{i_l}_0,\ldots,\b^{i_l}_{k_0-1}<\g_0\mbox{ and for all $j$: }\x_j<\b^{i_l}_{k_{j-1}},\ldots,\b^i_{k_j-1}<\g_j, l\in\{1,2\}.$$
Then since $$\tau(\alpha_0,\ldots,\alpha_n,\b^{i_l}_0,\ldots,\b^{i_l}_{k_0-1},\ldots,\b^{i_l}_{k_{s-1}},\ldots,\b^{i_l}_{k_s-1})\in B,$$ and $[B]^2\subseteq A$, we have 
$$\begin{array}{l}
\{\tau(\alpha_0,\ldots,\alpha_n,\b^{i_1}_0,\ldots,\b^{i_1}_{k_0-1},\ldots,
\b^{i_1}_{k_{s-1}},\ldots,\b^{i_1}_{k_s-1}),\\
\hspace{2mm}\tau(\alpha_0,\ldots,\alpha_n,\b^{i_2}_0,\ldots,\b^{i_2}_{k_0-1},\ldots,
\b^{i_2}_{k_{s-1}},\ldots,\b^{i_2}_{k_s-1})\}\in A
\end{array}$$
Hence
\begin{equation}\label{seq2}
(\b^{l_1}_0,\ldots,\b^{l_1}_{k_0-1},\b^{l_2}_0,\ldots,\b^{l_2}_{k_0-1},\ldots,
\b^{l_1}_{k_{s-1}},\ldots,\b^{l_1}_{k_s-1},\b^{l_2}_{k_{s-1}},\ldots,\b^{l_2}_{k_s-1})\in C
\end{equation} contrary to the assumption $[B_1]^{2k_0}\times...\times [B_s]^{2k_s}\cap C=\emptyset$. We have proved the claim.

Now we define
\begin{equation}
\begin{array}{l}
B^*=\{\tau(\alpha_0,\ldots,\alpha_n,\zeta^0_0,\ldots,\zeta^0_{k_0-1},\ldots,
\zeta^s_{k_{s-1}},\ldots,\zeta^s_{k_s-1}):\\
\hspace{15mm}(\zeta^0_0,\ldots,\zeta^0_{k_0-1})\in B_0^{k_0},
\ldots,
(\zeta^s_{k_{s-1}},\ldots,\zeta^s_{k_s-1})\in  B_s^{k_s}\}.\end{array}
\end{equation} Then $B^*\in L, |B^*|=\aleph_1$ and $[B^*]^2\subseteq A$.  
\end{proof}

What if we do not assume $0^\sharp$? We show that if we start from $L$ and use forcing we can obtain a model in which $C(Q^{\MMa,2}_{\omega_1})\ne L$. 

\begin{theorem}
If Con(ZF), then Con({\zfc}+$C(Q^{\MMa,2}_{\omega_1})\ne L$).
\end{theorem}
\begin{proof}
Assume $V=L$. Jensen and Johnsbr\aa ten \cite{MR0419229} define a sequence $T_n$ of Souslin trees in $L$ and a CCC forcing notion $\oP$ which forces the set $a$ of $n$ such that $\check{T_n}\mbox{ is Souslin}$ to be non-constructible. But $a\in C(Q^{\MMa,2}_{\omega_1})$
since the trees $T_n$ are in $C(Q^{\MMa,2}_{\omega_1})$ and Sousliness of a tree can be expressed in $\L(Q^{\MMa,2}_{\omega_1})$ by \cite[page 223]{MR0453484}.  So we are done.
\end{proof}

This result can be strengthened in a number of ways.  In \cite{MR1197203} an $\omega_1$-sequence of Souslin trees is constructed from $\diamondsuit$ giving rise to forcing extensions in which $\L(Q^{\MMa,2}_{\omega_1})$ can express some ostensibly second order properties, and 
$C(Q^{\MMa,2}_{\omega_1})$ is very different from $L$.

There are several stronger versions of $Q^{\MMa,<\omega}_\kappa$, for example
$$Q^{\mbox{\tiny MR}}_\kappa x_1,x_2,x_3\psi(x_1,x_2,x_3)\iff$$
$$\exists \mathcal{X}(\forall X_1,X_2\in\mathcal{X})(\forall x_1,x_2\in X_1)(\forall x_3\in X_2)\psi(x_1,x_2,x_3,\vec{y}),$$ where $X_1,X_2$ range over sets of size $\kappa$ and $\mathcal{X}$ ranges over families of size $\kappa$ of sets of size $\kappa$ (\cite{MR573950}). The above is actually just one of the various forms of similar quantifiers that $\L(Q^{\mbox{\tiny MR}}_{\kappa})$ has. The logic $\L(Q^{\mbox{\tiny MR}}_{\aleph_1})$ is still countably compact assuming $\diamondsuit$. We do not know whether $0^\sharp$ implies $C(Q^{\mbox{\tiny MR}}_{\kappa})=L$.

\section{The Cofinality Quantifier}

The cofinality quantifier of Shelah \cite{MR0376334} says that a given linear order has cofinality $\kappa$. Its main importance lies in the fact that it satisfies the compactness theorem irrespective of the cardinality of the vocabulary. Such logics are called {\em fully compact}. This logic has also a natural complete axiomatization, provably in {\zfc}. This makes the cofinality quantifier particularly appealing in this project, even though we do not have a clear picture yet of the connection between model theoretic properties of logics $\L^*$and set theoretic properties of $C(\L^*)$.

The cofinality quantifier $Q^{\cf}_\kappa$ for a regular $\kappa$ is defined as follows:
\begin{eqnarray*}
\mm\models Q^{\cof}_\kappa xy\phi(x,y,\vec{a})&\iff&\{(c,d):\mm\models\phi(c,d,\vec{a})\}\\
&&\mbox{ is a linear order of cofinality $\kappa$.}
\end{eqnarray*}



We will denote by ${C^*_{\kappa}}$ the inner model $C(Q^{\cof}_\kappa)$. Note  that $\cofmodel{\kappa}$ need not compute cofinality $\kappa$ correctly, it just knows which ordinals have cofinality $\kappa$ in $V$. The model knows this as if the model had an oracle for
 exactly this but nothing else. Thus while many more ordinals may have cofinality $\kappa$ in $V$ than in ${C^*_{\kappa}}$, still the property of an ordinal having cofinality $\kappa$ in $V$ is recognised in ${C^*_{\kappa}}$ in the sense that for all $\beta$ and $A,R\in {C^*_{\kappa}}$:
\begin{itemize}
 \item $\{\alpha<\beta : \cof^V(\alpha)=\kappa\}\in {C^*_{\kappa}}$
 \item $\{\alpha<\beta : \cof^V(\alpha)\ne\kappa\}\in {C^*_{\kappa}}$ 
 \item $\{\alpha<\beta : \cof^V(\alpha)=\kappa\iff \cof^{{C^*_{\kappa}}}(\alpha)=\kappa\}\in {C^*_{\kappa}}$ 
 \item $\{a\in A : \{(b,c):(a,b,c)\in R\}$  is a linear order on $A$ 
 with cofinality (in $V$) 
 equal to $\kappa\}\in {C^*_{\kappa}}$.
\end{itemize}

Let $\On_\kappa$ be the class of ordinals of cofinality $\kappa$.  Let $L(\On_\kappa)$ be $L$ defined in the expanded language $\{\in, \On_\kappa\}$. Now $L(\On_\kappa)\subseteq \cofmodel{\kappa}$ because we can use the equivalence of $\On_\kappa(\beta)$ with 
$Q^{\cof}_\kappa xy(x\in y\wedge y\in \beta)$. Conversely,
$\cofmodel{\kappa}\subseteq L(\On_\kappa)$ because if $E$ is a club of $\beta$ such that  for every linear order $R\in L_\beta(\On_\kappa)$ there is an ordinal $\gamma<\beta$ and a function $f\in L_\beta(\On_\kappa)$ mapping  $\gamma$ cofinally into $R$, then $L'_\alpha\subseteq L_\beta(\On_\kappa)$ whenever $\alpha\le\beta\in E$. 
We have proved $$\cofmodel{\kappa}=L(\On_\kappa).$$ 

We use ${C^*}$ to denote ${C^*_{\omega}}$. 

The following related quantifier turns out to be useful, too:
\begin{eqnarray*}
\mm\models Q^{\cof}_{<\kappa} xy\phi(x,y,\vec{a})&\iff&\{(c,d):\mm\models\phi(c,d,\vec{a})\}\\
&&\mbox{ is a linear order of cofinality $<\kappa$.}
\end{eqnarray*}
We use 
${C^*_{\kappa,\l}}$ to denote $C(Q^{\cof}_{\kappa},Q^{\cof}_{\l})$ and
 ${C^*_{<\kappa}}$ to denote $C(Q^{\cof}_{<\kappa})$. Respectively, ${C^*_{\le\kappa}}$
 denotes $C(Q^{\cof}_{\le\kappa})$.

Our results show that the inner models ${C^*_{<\kappa}}$ all resemble ${C^*}$ in many ways (see e.g. Theorem~\ref{515}), and accordingly we indeed focus mostly on ${C^*}$.

The logics $C^*_{\kappa,\l}$ and ${C^*_{<\kappa}}$ are adequate to truth in themselves (recall Definition~\ref{4r67}), with $\kappa,\lambda$ as parameters, whence these inner models satisfy AC. 

We can translate the formulas $\Phi_{\L(Q^{\cof}_{\kappa})}(x)$ and $\Psi_{\L(Q^{\cof}_{\kappa})}(x,y)$, introduced in Proposition~\ref{wpoa}, into  $\hat{\Phi}_{\L(Q^{\cof}_{\kappa})}(x,\kappa)$ and $\hat{\Psi}_{\L(Q^{\cof}_{\kappa})}(x,y,\kappa)$
in the first order language of set theory by systematically replacing
$$Q^{\cof}_\kappa xy\phi(x,y,\vec{a})$$
by 
the canonical set-theoretic formula saying the same thing.
Then
 for all $M$ with $\a=M\cap\On$ and $a,b\in M$:
 
\begin{enumerate}
\item $\hat{\Phi}_{\L(Q^{\cof}_{\kappa})}(a,\kappa)\leftrightarrow[(M,\in)\models\Phi_{\L(Q^{\cof}_{\kappa})}(a)]\leftrightarrow a\in C^*_\kappa$.
\item $\hat{\Psi}_{\L(Q^{\cof}_{\kappa})}(a,b,\kappa)\leftrightarrow[(M,\in)\models\Psi_{\L(Q^{\cof}_{\kappa})}(a,b)]\leftrightarrow a<'_\a b$.
\end{enumerate}

\begin{lemma}\label{nice}
If $M_1$ and $M_2$ are two transitive models of {\zfc} such that
for all
$\alpha$: 
$$M_1\models\cof(\alpha)=\kappa\iff M_2\models\cof(\alpha)=\kappa,$$ then $$(\cofmodel{\kappa})^{M_1}=(\cofmodel{\kappa})^{M_2}.$$ 

\end{lemma}

\begin{proof}
Let $(L'_\a)$ be the hierarchy defining $(\cofmodel{\kappa})^{M_1}$ and 
$(L''_\a)$ be the hierarchy defining $(\cofmodel{\kappa})^{M_2}$. By induction, $L'_\a=L''_\a$ for all $\a$. 
\end{proof}

By letting $M_2=V$ in Proposition~\ref{nice}  we get

\begin{corollary}
Suppose $M$ is a transitive model of {\zfc} such that for all $\a$: $$\cof(\alpha)=\kappa\iff M\models\cof(\alpha)=\kappa,$$ then $$(\cofmodel{\kappa})^{M}=\cofmodel{\kappa}.$$ 

\end{corollary}

This is a useful criterion. Note that $(\cofmodel{\kappa})^{M}\ne\cofmodel{\kappa}$ is a perfectly possible situation: In Theorem~\ref{2479} below we construct a model $M$ in which CH is false in $C^*$. 
So $(C^*)^M\ne L$. Thus in $M$ it is true 
that ${(C^*)}^{L}\ne{C^*}$. 
${(\cofmodel{\kappa})}^{M}\ne\cofmodel{\kappa}$ also if $\kappa=\omega$, $V=L^\mu$ and $M=C^*$
 (see the below Theorem~\ref{wepe}). In this respect $\cofmodel{\kappa}$ resembles $\hod$. There are other respects in which $\cofmodel{\kappa}$ resembles $L$.

\begin{lemma}
Suppose $(L'_\a)$ is the hierarchy forming $\cofmodel{\kappa}$. Then for $\a<\kappa$
we have $L'_\a=L_\a$.\end{lemma}

We can relativize ${C^*}$ to a set $X$ of ordinals as follows. Let us define a new generalized quantifier as follows:

\begin{eqnarray*}
\mm\models Q_X xy\phi(x,y,\vec{a})&\iff&\{(c,d):\mm\models\phi(c,d,\vec{a})\}\\
&&\mbox{ is a well-order of type $\in X$.}
\end{eqnarray*}

\noindent We define $C^*\hspace{-.5mm}(X)$ as $C(Q^{\cof}_\omega,Q_X)$. Of course, $C^*\hspace{-.5mm}(X)=L(\On_\omega,X)$.


We will prove a stronger form of the next Proposition in the next Theorem, but we include this here for completeness:

\begin{proposition}\label{cof}
If $0^\sharp$ exists, then $0^\sharp\in C(Q^{\cf}_\kappa)$.
\end{proposition} 

\begin{proof}Let $I$ be the canonical set of indiscernibles obtained from $0^\sharp$.
Let us first prove that ordinals $\xi$ which are regular cardinals in $L$ and have cofinality $>\omega$ in $V$ are in $I$. Suppose $\xi\notin I$. Note that $\xi>\min(I)$. Let $\delta$ be the largest element of $I\cap\xi$. Let $\lambda_1<\lambda_2<...$ be an infinite sequence of elements of $I$ above $\xi$. Let $$\tau_{n}(x_1,\ldots,x_{k_n}), n<\omega,$$ be a list of all the Skolem terms of the language of set theory relative to the theory $\zfc+V=L$. If $\alpha<\xi$, then $$\alpha=\tau_{n_\alpha}(\gamma_1,\ldots,\gamma_{m_n},\lambda_1,\ldots,\lambda_{l_n})$$ for some $\gamma_1,\ldots,\gamma_{m_n}\in I\cap\delta$ and some $l_n<\omega$. Let us fix $n$ for a moment and consider the set $$A_n=\{\tau_{n_\alpha}(\beta_1,\ldots,\beta_{m_n},\lambda_1,\ldots,\lambda_{l_n}): \beta_1,\ldots,\beta_n<\delta\}.$$ Note that $A_n\in L$ and $|A_n|^L\le|\delta|^L<\xi$, because $\xi$ is a cardinal in $L$. Let $\eta_n=\sup(A_n)$. Since $\xi$ is regular in $L$, $\eta_n<\xi$. Since $\xi$ has cofinality $>\omega$ in $V$, $\eta=\sup_n\eta_n<\xi$. But we have now proved that every $\alpha<\xi$ is below $\eta$, a contradiction. So we may conclude that necessarily $\xi\in I$.

Suppose now $\kappa=\omega$. Let 
$$X=\{\xi\in L'_{\aleph_\omega}: (L'_{\aleph_\omega},\in)\models ``\xi\mbox{ is regular  in $L$}"\wedge \neg Q^{cf}_\kappa xy(x\in y\wedge y\in \xi)\}$$
Now $X$ is an infinite subset of $I$ and $X\in  C(Q^{\cf}_\kappa)$. Hence $0^\sharp
\in C(Q^{\cf}_\kappa)$: 
$$0^\sharp=\{\ulcorner\phi(x_1,\ldots,x_n)\urcorner:(L_{\aleph_\omega},\in)
\models\phi(\gamma_1,\ldots,\gamma_n)\mbox{ for some $\gamma_1<...<\gamma_n$ in $X$}\}.$$ 
If $\kappa=\aleph_\alpha>\omega$, then we use
$$X=\{\xi\in L'_{\aleph_{\alpha+\omega}}: (L'_{\aleph_{\alpha+\omega}},\in)\models ``\xi\mbox{ is regular  in $L$}"\wedge  Q^{cf}_\kappa xy(x\in y\wedge y\in \xi)\}$$
and argue as above that $0^\sharp\in 
C(Q^{\cf}_\kappa)$. \end{proof}

More generally, the above argument shows that $x^\sharp\in C^*\hspace{-.5mm}(x)$ for any $x$ such that $x^\sharp$ exists. Hence ${C^*}\ne L(x)$ whenever  $x$ is a set of ordinals such that $x^\sharp$ exists in $V$ (see Theorem~\ref{2976543}).

\begin{theorem}\label{2976543}
Exactly one of the following always holds:

\begin{enumerate}
\item $C^*$ is closed under sharps, (equivalently, $x^\sharp$ exists for all $x\subset On$ such that $x\in C^*$).
\item $C^*$ is not closed under sharps and moreover $C^*= L(x)$ for some set $x\subset On$. (Equivalently, there is $x\subset On$ such that $x\in C^*$ but $x^\sharp$ does not exist.)
\end{enumerate}
 
\end{theorem}

\begin{proof}

Suppose (1) does not hold. Suppose $a\subseteq \l$, $\l>\omega_1$, such that $a\in C^*$ but $a^\sharp$ does not exist. Let $S=\{\a<\l^+ : \cof^V(\a)=\omega\}$. We show that $C^*=L(a,S)$. Trivially,  
$C^*\supseteq L(a,S)$. For $C^*\subseteq L(a,S)$ it is enough to show that one can detect in $L(a,S)$ whether a given $\delta\in On$
has cofinality $\omega$ (in $V$) or not. If $\cof(\delta)=\omega$, and $c\subseteq\delta$ is a cofinal $\omega$-sequence in $\delta$, then the Covering Theorem for $L(a)$ gives a set $b\in L(a)$ such that $c\subseteq b\subset\lambda$, $\sup(c)=\sup(b)$ and $|b|=\l$. The order type of $b$ is in $S$. Hence whether $\delta$ has cofinality $\omega$ or not can be detected in $L(a,S)$.
\end{proof}

\begin{corollary}
 If $x^\sharp$  does not exist for some $x\in C^*$, then there is $\l$ such that $C^*\models 2^\kappa=\kappa^+$ for all $\kappa\ge\l$.
\end{corollary}

\begin{theorem}\label{core}
The Dodd-Jensen Core model is contained in ${C^*}$.
\end{theorem}

\begin{proof}Let $K$ be the Dodd-Jensen Core model of $C^\ast$. We show that $K$ is the core model of $V$.
 Assume otherwise and let $M_0$ be the minimal Dodd-Jensen mouse missing from $K$. (Minimality here means in the canonical pre-well ordering of mice.)
  Let $\kappa_0$ be the cardinal of $M_0$ on which $M_0$ has the $M_0$-normal measure. Denote this normal measure by $U_0$. Note that $M_0=J_\alpha^{U_0}$
  for some $\alpha $. $J_\alpha[U_0]$ is the Jensen $J$-hierarchy of constructibility from $U_0$, where $J_\alpha[U_0]= \bigcup_{\beta<\omega\alpha }
  S_\beta^{U_0}$,  where $S_\beta ^{U_0}$ is the finer $S$-hierarchy.

   Let $\xi_0 $ be  $(\kappa_0^+)^{M_0}$.  (If $(\kappa_0^+)^{M_0}$ does not exist in $M_0$ put $\xi_0=M\cap ON$.). Let $\delta=\cof^V(\xi_0)$.

   For an ordinal $\beta$ let $M_\beta$ be the $\beta$'th iterated ultrapower of $M_0$ where for $\beta\leq\gamma$ let
   $j_{\beta,\gamma}:M_\beta\rightarrow
   M_\gamma$ be the canonical ultrapower embedding. $j_{\beta,\gamma}$ is a $\Sigma_0$-embedding. Let $\kappa_\beta=j_{0\beta}(\kappa_0),
   U_\beta=j_{0\beta}(U_0),\xi_\beta=j_{0\beta}(\xi_0)$. (In case $(\kappa_0^+)^{M_0}$ does not exist we put $\xi_\beta=M_\beta\cap ON $.).
   $\kappa_\beta$ is the critical point of $j_{\beta\gamma}$ for $\beta<\gamma$.  For a limit $\beta$ and $A\in M_\beta,A\subseteq \kappa_\beta$ $A\in
   U_\beta$ iff $\kappa_\gamma\in A$ for large enough $\gamma<\beta$.

   \begin{claim} 1. For every $\beta$ we have  $\xi_\beta=\sup j"_{0\beta}(\xi_0)$. Hence $\cof^V(\xi_\beta)=\delta$.
   \end{claim}
   \begin{proof} Every $\eta<\xi_\beta$ is of the form $j_{0\beta}(f)(\kappa_{\gamma_0}\ldots,\kappa_{\gamma_{n-1}})$ for some
   $\gamma_0<\gamma_1\ldots<\gamma_{n-1}<\beta$ and for some $f\in M_0,f:\kappa_0^n\rightarrow\xi_0$. By definition of $\xi_0$ there is $\rho<\xi_0$ such that
   $f(\alpha_0,\ldots\alpha_{n-1})<\rho$ for every $\langle \alpha_0,\ldots\alpha_{n-1}\rangle\in \kappa_0^n$. Hence it follows that every value of
   $j_{0\beta}(f)$ is bounded by $j_{0\beta}(\rho)$. So $\eta<j_{0\beta}(\rho)$, which proves the claim.
   \end{proof}

   The usual proof of GCH in $L[U]$ shows that  $\kappa_\beta^{\kappa_\beta}\cap M_\beta\subseteq J_{\xi_\beta}^{U_\beta}$ and that $J^{U_\beta}_{\xi_\beta}$ is
   the increasing union of $\delta$ members of $M_\beta$, each one having cardinality $\kappa_\beta$ in $M_\beta$.

   \begin{claim} 2. Let $\kappa_0<\eta<\kappa_\beta$ be such that $M_\beta\models \mbox{$\eta$ is regular}$, then either there is $\gamma<\beta $
   such that $\eta=\kappa_\gamma$ or $\cof^V(\eta)=\delta$.
   \end{claim}
   \begin{proof} By induction on $\beta$. The claim is vacuously true for $\beta=0$. For $\beta$ limit $\kappa_\beta=\sup\{\kappa_\gamma|\gamma<\beta\}$. Hence
   there is $\alpha<\beta $ such that $\eta<\kappa_\alpha$. $j_{\alpha\beta}(\eta)=\eta) $ so $M_\alpha\models\eta\mbox{ is regular}$. So the claim in this
   case follows from the induction assumption.

   We are left with the case that $\beta=\alpha+1$. If $\eta\leq\kappa_\alpha$ the claim follows from the inductive assumption for $\alpha$ as in the
   limit case. So we are left with the case $\kappa_\alpha<\eta<\kappa_\beta$. $M_\beta$ is the ultrapower of $M_\alpha$ by $U_\alpha$, so $\eta$ is
   represented in this ultrapower by a function $f\in M_\alpha$ whose domain is $\kappa_\alpha$. By the assumption
   $\eta<\kappa_\beta=j_{\alpha\beta}(\kappa_\alpha)$ we can assume $f(\rho)<\kappa_\alpha$ for every $\rho<\kappa_\alpha$. By the assumption
   $\kappa_\alpha<\eta$ we can assume that $\rho<f(\rho)$ for every $\rho<\kappa_\alpha$ and by the assumption that $\eta$ is regular in $M_\beta$ we can
   assume that $f(\rho)$ is regular in $M_\alpha$ for every $\rho<\kappa_\alpha$. In order to simplify notation put
   $M=M_\alpha,\kappa=\kappa_\alpha,U=U_\alpha, \mbox{ and }\xi=\xi_\alpha$.

   In order to show that $\cof^V(\eta)=\delta$ we shall define (in $V$)   a sequence $\langle g_\nu|\nu<\delta \rangle$ of functions in $\kappa^\kappa\cap M$  such
   that :
   \begin{enumerate}
     \item The sequence is increasing modulo $U$.
     \item For every $\rho<\kappa$, $g_\nu (\rho)<f(\rho)$.
     \item The ordinals represented by these functions in the ultrapower of $M$ by $U$ are cofinal in $\eta$.
   \end{enumerate}

 By the definition of $\xi$ and the previous claim  we can represent $\kappa^\kappa\cap M$ as an increasing union $\bigcup_{\psi<\delta} F_\psi$ where for
 every $\psi<\delta$, $F_\psi\in M$ and $F_\psi$ has cardinality $\kappa$ in $M$. For $\psi<\delta$ fix an enumeration in $M$ of $\langle h^\psi_\rho
 |\rho<\kappa \rangle $ of the set $G_\psi=\{h\in F_\psi | \forall\rho<\kappa (h(\rho)<f(\rho)) \}$. Let $f_\psi\in \kappa^\kappa$ be defined by
 $f_\psi(\rho)=\sup(\{h^\psi_\mu(\rho) |\mu<\rho\})$. Clearly $f_\psi\in M$ and $f_\psi$ bounds all the functions in $G_\psi$ modulo $U$. Also since for
 all $\rho<\kappa$ and $h\in G_\psi$ $h(\rho)<f(\rho)$ we obtain $f_\psi(\rho)<f(\rho)$. (Recall that $f(\rho)>\rho$, $f(\rho)$ is regular in $M$ and
 $f_\psi(\rho)$ is the sup of a set in $M$ whose cardinality in $M$ is $\rho$. ).

 Define $g_\nu$ by induction on $\nu<\delta$. By induction we shall also define an increasing sequence $\langle \psi_\nu |\nu<\delta\rangle $ such that
 $\psi_\nu<\delta$ and $g_\nu\in G_{\psi_\nu}$. Given $\langle \psi_\mu | \mu <\nu \rangle $ let $\sigma$ be their sup. Let $g_\nu$ be $f_\sigma$ and let
 $\psi_\nu$ be the minimal member of $\delta-\sigma$ such that $f_\sigma\in G_{\psi_\nu}$. The induction assumptions on $g_\mu,\psi_\mu$ for $\mu<\nu$
 and
 the properties of $f_\sigma$ yields that $g_\nu$ and $\psi_\nu$ also satisfy the required inductive assumption.

 The fact that the sequence of ordinals represented by $\langle g_\nu |\nu<\delta \rangle $ in the ultrapower of $M$ by $U$ is cofinal  in  $\eta$ follows
 from the fact that  every ordinal below $\eta$ is represented by some function $h$ which is bounded everywhere by $f$,  hence it belongs to $G_\psi$ for
 some $\psi<\delta$. There is $\nu$ such that $\psi<\psi_\nu$ and then $g_{\nu+1}$ will bound $h$ modulo $U$.
 \end{proof}

 The minimality of $M_0$ (hence the minimality of the equivalent $M_\beta$) implies that for every $\beta$, $\mathcal{P}(\kappa_\beta)\cap
 K=\mathcal{P}(\kappa_\beta)\cap M_\beta$.  It follows that  $\rho\leq\kappa_\beta$ is regular in $K$ iff  it is regular in $M_\beta$. In particular for every $\beta$, $\kappa_\beta$ is regular in $K$ since it is regular in $M_\beta$.
 
 \begin{claim} 3. Let $\lambda$ be a regular cardinal greater than $\max(|M_0|,\delta)$.  Then there there is $D\in C^\ast$,  $D\subseteq E=
 \{\kappa_\beta|\beta<\lambda\}$  which is cofinal in $\lambda$.
 \end{claim}
 \begin{proof} Note that $\lambda>|M_0|$ implies that the set $E=\{\kappa_\beta |\beta<\lambda\}$ is a club in $\lambda$.   Let $S^\lambda_0$ be the set of ordinals in $\lambda$ whose cofinality (in $V$ ) is $\omega$. Obviously both $E-S^\lambda_0$ and
 $E\cap S^\lambda_0$ are unbounded in $\lambda$. Let $C$ be the set of the ordinals of $\lambda-\kappa_0$ which are regular in $K$. By the definition of $C^\ast$
 and $K$ both $S^\lambda_0$ and $C$ are in $C^\ast$. Also $E\subseteq C$ since $\kappa_\beta$ is regular in $M_\beta$, hence regular in $K$.

 If $\delta\neq\omega$ then we can take $D=C\cap S^\lambda_0$ which by Claim 2  is a subset of $E$ which is unbounded in $\lambda$. If
 $\delta=\omega$ then similarly we can take  $D=C - S^\lambda_0$.  In both cases $D\in C^\ast$.
 \end{proof}

 Pick $\lambda,E$ as in the Claim above and let $D\subseteq\lambda$ be the witness to the claim. It is well known that  for every $X\in M_\lambda$
 $X\in
 U_\lambda$ iff $X\subseteq\lambda$ and $X$ contains a final segment of $E$. Since $U_\lambda$ is an ultrafilter on $\lambda$ in $M_\lambda$ we get that
 for
 $X\in M_\lambda,  X\subseteq\lambda$ $X\in U_\beta$ iff $X$ contains a final segment of $D$. Let $F_D$ be the filter on $\lambda$ generated by final
 segments of $D$. $D\in C^\ast$ implies that $L(F_D)\subseteq C^\ast$. $M_\lambda=J^{U_\lambda}_\alpha$ for some ordinal $\alpha$. But since
 $U_\lambda=F_D\cap M_\lambda$ we get that $M_\lambda=J^{F_D}_\alpha$. Now this implies that $M_\lambda\in C^\ast$. This is because $C^\ast$ contains an iterate of the
 mouse $M_0$ and then by standard Dodd-Jensen Core model techniques $M_0\in C^\ast$, which is clearly a contradiction.
 \end{proof}

\begin{theorem}\label{lofmu}
Suppose  an inner model with a measurable cardinal exists. Then ${C^*}$ contains some inner model $L^\nu$ for a measurable cardinal.
\end{theorem}


\begin{proof}This is as the proof of Theorem~\ref{core}. 
Suppose $L^\mu$ exists, but does not exist in $C^*$.
  Let $\kappa_0$ be the cardinal of $M_0=L^\mu$ on which $L^\mu$ has the normal measure. Denote this normal measure by $U_0$.    Let $\xi_0 $ be  $(\kappa_0^+)^{M_0}$  and let $\delta=\cof^V(\xi_0)$.
  
   For an ordinal $\beta$ let $M_\beta$ be the $\beta$'th iterated ultrapower of $M_0$ and for $\beta\leq\gamma$ let
   $j_{\beta,\gamma}:M_\beta\rightarrow
   M_\gamma$ be the canonical ultrapower embedding. $j_{\beta\gamma}$ is a $\Sigma_0$-embedding. Let $\kappa_\beta=j_{0\beta}(\kappa_0),
   U_\beta=j_{0\beta}(U_0),\xi_\beta=j_{0\beta}(\xi_0)$. 
   $\kappa_\beta$ is the critical point of $j_{\beta\gamma}$ for $\beta<\gamma$.  For a limit $\beta$ and $A\in M_\beta,A\subseteq \kappa_\beta$, $A\in
   U_\beta$ iff $\kappa_\gamma\in A$ for large enough $\gamma<\beta$.

   \begin{claim} 1. For every $\beta$, $\xi_\beta=\sup j_{0\beta}"(\xi_0)$. Hence $\cof^V(\xi_\beta)=\delta$.
   \end{claim}
   \begin{proof} Every $\eta<\xi_\beta$ is of the form $j_{0\beta}(f)(\kappa_{\gamma_0}\ldots,\kappa_{\gamma_{n-1}})$ for some
   $\gamma_0<\gamma_1\ldots<\gamma_{n-1}<\beta$ and for some $f\in M_0,f:\kappa_0^n\rightarrow\xi_0$. By definition of $\xi_0$ there is $\rho<\xi_0$ such that
   $f(\alpha_0,\ldots\alpha_{n-1})<\rho$ for every $\langle \alpha_0,\ldots\alpha_{n-1}\rangle\in \kappa_0^n$. Hence it follows that every value of
   $j_{0\beta}(f)$ is bounded by $j_{0\beta}(\rho)$. So $\eta<j_{0\beta}(\rho)$, which proves the claim.
   \end{proof}

   The usual proof of GCH in $L[U]$ shows that  $\kappa_\beta^{\kappa_\beta}\cap M_\beta\subseteq J_{\xi_\beta}^{U_\beta}$ and that $J^{U_\beta}_{\xi_\beta}$ is
   the increasing union of $\delta$ members of $M_\beta$, each one having cardinality $\kappa_\beta$ in $M_\beta$.

   \begin{claim} 2. Let $\kappa_0<\eta<\kappa_\beta$ be such that $M_\beta\models \mbox{$\eta$ is regular}$, then either there is $\gamma<\beta $
   such that $\eta=\kappa_\gamma$ or $\cof^V(\eta)=\delta$.
   \end{claim}
   
   \begin{proof} By induction on $\beta$. The claim is vacuously true for $\beta=0$. For $\beta$ limit $\kappa_\beta=\sup\{\kappa_\gamma|\gamma<\beta\}$. Hence
   there is $\alpha<\beta $ such that $\eta<\kappa_\alpha$. $j_{\alpha\beta}(\eta)=\eta $ so $M_\alpha\models\eta\mbox{ is regular}$. So the claim in this
   case follows from the induction assumption.

   We are left with the case that $\beta=\alpha+1$. If $\eta\leq\kappa_\alpha$ the claim follows from the inductive assumption for $\alpha$ as in the
   limit case. So we are left with the case $\kappa_\alpha<\eta<\kappa_\beta$. $M_\beta$ is the ultrapower of $M_\alpha$ by $U_\alpha$, so $\eta$ is
   represented in this ultrapower by a function $f\in M_\alpha$ whose domain is $\kappa_\alpha$. By the assumption
   $\eta<\kappa_\beta=j_{\alpha\beta}(\kappa_\alpha)$ we can assume that
    $f(\rho)<\kappa_\alpha$  for every $\rho<\kappa_\alpha$. By the assumption
   $\kappa_\alpha<\eta$ we can assume that $\rho<f(\rho)$ for every $\rho<\kappa_\alpha$ and by the assumption that $\eta$ is regular in $M_\beta$ we can
   assume that $f(\rho)$ is regular in $M_\alpha$ for every $\rho<\kappa_\alpha$. In order to simplify notation put
   $M=M_\alpha,\kappa=\kappa_\alpha,U=U_\alpha, \mbox{ and }\xi=\xi_\alpha$.

   In order to show that $\cof^V(\eta)=\delta$ we shall define (in $V$)   a sequence $\langle g_\nu|\nu<\delta \rangle$ of functions in $\kappa^\kappa\cap M$  such
   that :
   \begin{enumerate}
     \item The sequence is increasing modulo $U$.
     \item For every $\rho<\kappa$, $g_\nu (\rho)<f(\rho)$.
     \item The ordinals represented by these functions in the ultrapower of $M$ by $U$ are cofinal in $\eta$.
   \end{enumerate}

 By the definition of $\xi$ and by the previous claim  we can represent $\kappa^\kappa\cap M$ as an increasing union $\bigcup_{\psi<\delta} F_\psi$ where for
 every $\psi<\delta$ $F_\psi\in M$ and $F_\psi$ has cardinality $\kappa$ in $M$. For $\psi<\delta$ fix an enumeration in $M$ of $\langle h^\psi_\rho
 |\rho<\kappa \rangle $ of the set $G_\psi=\{h\in F_\psi | \forall\rho<\kappa (h(\rho)<f(\rho)) \}$. Let $f_\psi\in \kappa^\kappa$ be defined by
 $f_\psi(\rho)=\sup(\{h^\psi_\mu(\rho) |\mu<\rho\})$. Clearly $f_\psi\in M$ and $f_\psi$ bounds all the functions in $G_\psi$ modulo $U$. Also because for
 all $\rho<\kappa$ and $h\in G_\psi$ $h(\rho)<f(\rho)$ we get that $f_\psi(\rho)<f(\rho)$. (Recall that $f(\rho)>\rho$, $f(\rho)$ is regular in $M$ and
 $f_\psi(\rho)$ is the sup of a set in $M$ whose cardinality in $M$ is $\rho$.).

 Define $g_\nu$ by induction on $\nu<\delta$. By induction we shall also define an increasing sequence $\langle \psi_\nu |\nu<\delta\rangle $ such that
 $\psi_\nu<\delta$ and $g_\nu\in G_{\psi_\nu}$. Given $\langle \psi_\mu | \mu <\nu \rangle $ let $\sigma$ be their sup. Let $g_\nu$ be $f_\sigma$ and let
 $\psi_\nu$ be the minimal member of $\delta-\sigma$ such that $f_\sigma\in G_{\psi_\nu}$. The induction assumptions on $g_\mu,\psi_\mu$ for $\mu<\nu$
 and
 the properties of $f_\sigma$ yields that $g_\nu$ and $\psi_\nu$ also satisfy the required inductive assumption.

 The fact that the sequence of ordinals represented by $\langle g_\nu |\nu<\delta \rangle $ in the ultrapower of $M$ by $U$ is cofinal  in  $\eta$ follows
 from the fact that  every ordinal bellow $\eta$ is represented by some function $h$ which is bounded everywhere by $f$, hence it belongs to $G_\psi$ for
 some $\psi<\delta$. There is $\nu$ such that $\psi<\psi_\nu$. Then $g_{\nu+1}$ will bound $h$ modulo $U$.
 \end{proof}

 We know already that $K\subseteq C^*$. Since
 $\mathcal{P}(\kappa_\beta)\cap
 K=\mathcal{P}(\kappa_\beta)\cap M_\beta$
 for every $\beta$,  it follows that  $\rho\leq\kappa_\beta$ is regular in $K$ iff  it is regular in $M_\beta$. In particular for every $\beta$, $\kappa_\beta$ is regular in $K$ since it is regular in $M_\beta$.
 
 \medskip
 
 \begin{claim} 3. Let $\lambda$ be a regular cardinal greater than $\max(|M_0|,\delta)$. Then there there is $D\in C^\ast$,  $D\subseteq E=
 \{\kappa_\beta|\beta<\lambda\}$  which is cofinal in $\lambda$.
 \end{claim}
\medskip

\noindent Proof of the Claim: Note that $\lambda>|M_0|$ implies that the set $E=\{\kappa_\beta |\beta<\lambda\}$ is a club in $\lambda$.   Let $S^\lambda_0$ be the set of ordinals in $\lambda$ whose cofinality (in $V$ ) is $\omega$. Obviously both $E-S^\lambda_0$ and
 $E\cap S^\lambda_0$ are unbounded in $\lambda$. Let $C$ be the set of the ordinals of $\lambda-\kappa_0$ which are regular in $K$. By definition of $C^\ast$
 and $K$ both $S^\lambda_0$ and $C$ are in $C^\ast$. Also $E\subseteq C$ since $\kappa_\beta$ is regular in $M_\beta$, hence regular in $K$.

 If $\delta\neq\omega$ then we can take $D=C\cap S^\lambda_0$ which by Claim 2  is a subset of $E$ which is unbounded in $\lambda$. If
 $\delta=\omega$ then similarly we can take  $D=C -  S^\lambda_0$.  In both cases $D\in C^\ast$.
The Claim is proved.
\medskip

 Pick $\lambda,E$ as in the Claim above and let $D\subseteq\lambda$ be the witness to the claim. It is well known that  for every $X\in M_\lambda$
 $X\in
 U_\lambda$ iff $X\subseteq\lambda$ and $X$ contains a final segment of $E$. Since $U_\lambda$ is an ultrafilter on $\lambda$ in $M_\lambda$ we get that
 for
 $X\in M_\lambda,  X\subseteq\lambda$ $X\in U_\beta$ iff $X$ contains a final segment of $D$. Let $F_D$ be the filter on $\lambda$ generated by final
 segments of $D$. $D\in C^\ast$ implies that $L(F_D)\subseteq C^\ast$. $M_\lambda=J^{U_\lambda}_\alpha$ for some ordinal $\alpha$. But since
 $U_\lambda=F_D\cap M_\lambda$ we get that $M_\lambda=J^{F_D}_\alpha$. Thus $M_\lambda\in C^\ast$, i.e. $C^\ast$ contains an iterate of $M_0$. Hence $C^*$ contains an inner model with a measurable cardinal.
  \end{proof}

Below (Theorem~\ref{wepe}) we will show that if $L^\mu$ exists, then $({C^*})^{L^\mu}$ can be obtained by adding to the $\omega^2$th iterate of $L^\mu$ the sequence $\{\kappa_{\omega\cdot n}: n<\omega\}$.

In the presence of large cardinals, even with just uncountably many measurable cardinals, we can separate ${C^*}$ from both $L$ and $\hod$.
We first observe that in the special case that $V={C^*}$, there cannot exist even a single measurable cardinal. The proof is similar to Scott's proof that measurable cardinals violate $V=L$:

\begin{theorem}\label{wpoi}
If there is a measurable cardinal $\kappa$, then $V\ne\cofmodel{\lambda}$ for all $\lambda<\kappa$.
\end{theorem}

\begin{proof}
Suppose $V=\cofmodel{\lambda}$ but $\kappa>\lambda$ is a measurable cardinal. Let $i:V\to M$ with critical point $\kappa$ and $M^{\kappa}\subseteq M$. 
Now $(\cofmodel{\lambda})^M=(\cofmodel{\lambda})^V=V$, whence $M=V$. This contradicts Kunen's result \cite{MR0311478} that there cannot be a non-trivial $i:V\to V$.
\end{proof}

We can strengthen this as follows, at least for $\lambda=\omega$. Recall that \emph{covering} is said to hold for a inner model $M$ if for every  set $X$  of ordinals there is a set $Y\supseteq X$ of ordinals  such that $Y\in M$ and $|Y|\le |X|+\aleph_1$. We can show that if there is a measurable cardinal, then not only $V\ne C^*$, but we do not even have covering for $C^*$:

\begin{theorem}
If there is a measurable cardinal then covering fails for $C^*$.
\end{theorem}

\begin{proof} Let $i:V\to M$ with critical point $\kappa$ and $M^{\kappa}\subseteq M$. As above,  $i$ is an embedding of $C^*$ into $C^*$. Let  $\kappa_n$ be $i^n(\kappa)$ and $\kappa_\omega= \sup_n \kappa_n$. Clearly $i(\kappa_\omega)=\kappa_\omega$ and there are no fixed points of $i$ on the interval $[\kappa,\kappa_\omega)$.
 We prove that covering fails for $C^*$ by showing that the singular cardinal $\kappa_\omega$ is regular in $C^*$. Assume otherwise. Then the cofinality  $\alpha$ of $\kappa_\omega$ in $C^*$ is, by elementarity, a fixed point of $i$. Hence  $\alpha<\kappa$.  Let $Z$ be  a be a subset of $\kappa_\omega$ in $C^*$ witnessing the fact that the cofinality of $\kappa_\omega$ in $C^*$ is  $\alpha$. W.l.o.g., $Z$ is the minimal such set in the canonical wellordering of $C^*$. Hence $i(Z)=Z$. Let $\delta=\sup(Z\cap\kappa)$. Since $\kappa$ is regular, $\delta<\kappa$. Hence $i(\delta)=\delta$.
Let $\delta^*$ be the minimal member of $Z$ above $\delta$. Then $\delta^*\ge\kappa$ and $i(\delta^*)=\delta^*$. But there are no fixed points of $i$ on the interval $[\kappa,\kappa_\omega)$. We have reached a contradiction.
\end{proof}

On the other hand we will now use known results to show that we cannot fail covering for $C^*$ without an inner model for a measurable cardinal. It is curious that covering for $C^*$ is in this way entangled with measurable cardinals.

\begin{theorem}
If there is no inner model with a measurable cardinal then covering holds for $C^*$. 
\end{theorem}

\begin{proof}
By Theorem~\ref{core}, $K
  \subseteq C^*$. If there is no  inner model with a measurable cardinal, then $K$ satisfies covering by \cite{MR661475}. Hence all the more we have covering for $C^*$.
\end{proof}

 Kunen \cite{MR0337603} proved  that if there are uncountably many measurable cardinals, then AC fails in Chang's model $C(\L_{\omega_1\omega_1})$. Recall that Chang's model contains ${C^*}$ and  ${C^*}$ {\em does} satisfy AC.

\begin{theorem}\label{knn}
If $\la \kappa_n : n<\omega\ra$ is any sequence of measurable cardinals (in $V$) $>\lambda$, then  $\la \kappa_n : n<\omega\ra\notin\cofmodel{\lambda}$ and  $\cofmodel{\lambda}\ne\hod$.
\end{theorem}

\begin{proof}
We proceed as in  Kunen's proof (\cite{MR0337603}) that  AC fails in the Chang model if there are uncountably many measurable cardinals, except that we only use infinitely many measurable cardinals. Suppose $\kappa_n$, $n<\omega$, are measurable $>\lambda$. 
Let $\mu=\sup_n\kappa_n$.
Let $\prec$ be the first  well-order of $\mu^\omega$ in $\cofmodel{\lambda}$ in the canonical well-order of $\cofmodel{\lambda}$. Suppose 
$\la \kappa_n : n<\omega\ra\in\cofmodel{\lambda}$. Then for some $\eta$  it is the $\eta$th element in the well-order $\prec$. By \cite[Lemma 2]{MR0337603} there are only finitely many measurable cardinals  $\xi$ such that $\eta$ is moved by the ultrapower embedding of a normal ultrafilter on $\xi$. Let $n$ be such that the ultrapower embedding $j:V\to M$ by the normal ultrafilter on $\kappa_n$ does not move $\eta$. Since $\kappa_n>\lambda$, $(\cofmodel{\lambda})^M=\cofmodel{\lambda}$. Since $\mu$ is a strong limit cardinal $>\lambda$, $j(\mu)=\mu$. Since the construction of $\cofmodel{\lambda}$ proceeds in $M$ exactly as it does in $V$, $j(\prec)$ is also in $M$ the first well-ordering of $\mu^\omega$ that appears in $\cofmodel{\lambda}$. Hence $j(\prec)=\prec$. Since $j(\eta)=\eta$, the sequence $\la \kappa_n : n<\omega\ra$ is fixed by $j$. But this contradicts the fact that $j$ moves $\kappa_n$. 

If the $\kappa_n$ are the first $\omega$ measurable cardinals above $\lambda$, then the sequence $\la\kappa_n : n<\omega\ra$ is in $\hod$ and hence $\cofmodel{\lambda}\ne\hod$.
\end{proof}

\begin{definition}
 The {\em weak Chang model} is the model $C^\omega_{\omega_1}=C(\L^\omega_{\omega_1\omega_1})$.
\end{definition}

We can make the following observations about the relationship between the weak Chang model and the (full) Chang model. The weak Chang model clearly contains $C^*$ and $L(\R)$, as it contains
$C(\L^\omega_{\omega_1\omega})$. It is a potentially interesting intermediate model between $L(\R)$ and the (full) Chang model. If there is a measurable Woodin cardinal, then the Chang model satisfies AD, whence the weak model cannot satisfy AC, as the even bigger (full) Chang model cannot contain a well-ordering of all the reals.

\begin{theorem}
\begin{enumerate}
 \item  If $V=L^\mu$, then $C^\omega_{\omega_1}\ne L(\R)$. 
\item If $V$ is the inner model for $\omega_1$ measurable cardinals, then 
$C^\omega_{\omega_1}\ne $ Chang model.\end{enumerate}
\end{theorem}

\begin{proof}
For (1), suppose $V=L^\mu$, where $\mu$ is a normal measure on $\kappa$. Let us first note that  all the reals are in $C^\ast$, because under the assumption $V=L^\mu$ all the reals are in the Dodd-Jensen core model, which by our Theorem 5.5 is contained in $C^\ast$. Thus all the reals are in $C^\omega_{\omega_1}$. Also under the same assumption we have in $L(\mathbb{R})$ a $\Sigma^1_3$ well ordering of the reals of order type $\omega_1$. Hence $L(\mathbb{R})=L(A)$ for some $A\subseteq\omega_1$. Suppose now  $C^\omega_{\omega_1}= L(\R)$. Then there is an $A\subseteq\omega_1$ such that
 $C^\omega_{\omega_1}=L(A)$. By Theorem~\ref{lofmu} there is  in $C^*$, hence in $C^\omega_{\omega_1}$,  an inner model $L^\nu$ with a measurable cardinal $\delta$. But $\kappa$ is in $L^\mu$ the smallest ordinal which is measurable in an inner model. Hence $\kappa\le\delta$ and $A\subseteq\delta$. 
 But by \cite{MR0143710} there cannot be an inner model with a  measurable cardinal  $\delta$ in $L(A)$, where $A\subseteq\delta$. Therefore we must have $C^\omega_{\omega_1}\ne  L(\R)$.

For (2), we commence by noting that in the inner model for $\omega_1$ measurable cardinals there is a $\Sigma^1_3$-well-order of $\R$  \cite{MR0299469}. By means of this well-order we can well-order the formulas of $\L_{\omega_1\omega_1}^\omega$ in the Chang model. In this way we can define a well-order of
$C^\omega_{\omega_1}$ in the Chang model. However, since we assume uncountably many measurable cardinals, the Chang model does not satisfy AC \cite{MR0337603} (see also Theorem~\ref{knn}). Hence it must be that $C^\omega_{\omega_1}\ne $ Chang model. 
\end{proof}

If there is a Woodin cardinal, then ${C^*}\ne V$ in the strong sense that $\aleph_1$ is a large cardinal in ${C^*}$. So not only are there countable sequences of measurable cardinals which are not in  ${C^*}$ but there are even reals which are not in ${C^*}$:

\begin{theorem}\label{mahlo}
If there is a Woodin cardinal, then $\omega_1$ is (strongly) Mahlo in ${C^*}$.
\end{theorem}

\begin{proof} 
 To prove that $\omega_1$ is strongly inaccessible in $C^*$ suppose 
 $\a<\aleph_1$ and $$f:\omega_1\to (2^\a)^{C^*}$$ is 1-1.
 Let $\lambda$ be Woodin,  $\Q_{<\lambda}$ the countable stationary tower forcing and $G$ generic for this forcing. 
In $V[G]$ there is $j:V\to M$ such that $V[G]\models M^{\omega}\subset M$ and $j(\omega_1)=\lambda$. Thus $$j(f):\l\to ((2^\a)^{C^*})^{M}.$$
Let $a=j(f)(\omega_1^V)$. If $a\in V$, then $j(a)=a$, whence, as $a$ i.e. $j(a)$ is in the range of $j(f)$,  $a=f(\delta)$ for some $\delta<\omega_1$. But then 
$$a=j(a)=j(f)(j(\delta))=j(f)(\delta),$$
contradicting the fact that $a=j(f)(\omega_1)$. Hence $a\notin V$. However, $$(C^*)^M=(C^*_{<\l})^V,$$ since by general properties of this forcing, an ordinal has cofinality $\omega$ in $M$ iff it has cofinality $<\l$ in $V$. Hence $a\in C^*_{<\l}\subseteq V$, a contradiction.

To see that $\omega_1$ is Mahlo in ${C^*}$, suppose $D$ is a club on $\omega_1^V$, $D\in {C^*}$. Let $j$ and $M$ be as above. Then $j(D)$ is a club on $\lambda$ in $({C^*})^M$. Since $\omega_1^V$ is the critical point of $j$, $j(D)\cap\omega_1^V=D$. Since $j(D)$ is closed, $\omega_1^V\in j(D)$.  

\end{proof}

\begin{remark}
 In the previous  theorem we can replace the assumption of a Woodin cardinal by $\MM^{++}$.
\end{remark}

For cardinals $>\omega_1$ we have an even better result:

\begin{theorem}
Suppose there is a Woodin cardinal $\l$. Then every regular cardinal $\kappa$ such that $\omega_1<\kappa<\l$ is  weakly compact in $C^*$.
\end{theorem}

\begin{proof} Suppose $\l$ is a Woodin cardinal,  $\kappa>\omega_1$ is regular and $<\l$.  To prove that $\kappa$ is strongly inaccessible in $C^*$ we  use the ``$\le\omega$-closed" stationary tower forcing from \cite[Section 1]{MR1359154}. With this forcing, cofinality $\omega$ is not changed, whence $(C^*)^M=C^*$, so the proof of Theorem~\ref{mahlo} can be repeated mutatis mutandis. Thus we need only prove the tree property.
 Let the forcing, $j$ and $M$ be as above, in Theorem~\ref{mahlo}, with $j(\kappa)=\l$.  Suppose $T$ is a $\kappa$-tree in $C^*$. Then $j(T)$ is a $\l$-tree in $(C^*)^M=C^*$. We may assume $j(T\restriction\kappa)=T\restriction\kappa$. Let $t\in j(T)$ be of height $\kappa$ and $b=\{u\in j(T) : u<t\}=\{u\in T : u<t\}$. Now $b$ is a $\kappa$-branch of $T$ in $C^*$.
\end{proof}

As a further application of  $\omega$-closed stationary tower forcing we extend the above result as follows:

\begin{theorem}\label{ind}
If there is a proper class of Woodin cardinals, then the regular cardinals $\ge \aleph_2$ are  indiscernible\footnote{The cardinals are indiscernible even if the quantifier $Q^{\cof}_\omega$ is added to the language of set theory.} 
in ${C^*}$.
\end{theorem}
\def\bl{\bar{\lambda}}

\begin{proof}
We use the $\omega$-closed stationary tower  forcing of \cite{MR1359154}. Let us first prove an auxiliary claim:

\medskip

\noindent{\bf Claim 1:} If $\gl_1<\ldots<\gl_k$ and $\bl_1<\ldots<\bl_k$ are Woodin cardinals, and $\beta_1,\ldots,\beta_l<\min(\gl_1,\bl_1)$, then 
$${C^*}\models\Phi(\beta_1,\ldots,\beta_l,\gl_1,\ldots,\gl_k)\leftrightarrow
\Phi(\beta_1,\ldots,\beta_l,\bl_1,\ldots,\bl_k)$$ for all formulas $\Phi(x_1,\ldots,x_l,y_1,\ldots,y_k)$ of set theory. 
\medskip

To prove Claim 1, assume w.l.o.g. $\bl_1>\gl_1$. The proof proceeds by induction on $k$. The case $k=0$ is clear. Let us then assume the claim for $k-1$. Let $G$ be generic for the $\le\omega$-closed stationary tower forcing of \cite[Section 1]{MR1359154} with the generic embedding $$j:V\to M, M^\omega\subseteq M, j(\gl_1)=\bl_1, j(\bl_i)=\bl_i\mbox{ for $i>1$}.$$ A special feature of the $\omega$-closed stationary tower forcing of \cite{MR1359154} is that it does not introduce new ordinals of cofinality $\omega$. Thus $${{C^*}}^{V}={{C^*}}^{V[G]}={{C^*}}^{M}.$$
Suppose now $${C^*}\models\Phi(\beta_1,\ldots,\beta_l,\gl_1,\gl_2,\ldots,\gl_k).$$ By the induction hypothesis, in $M$, applied to $\gl_2,\ldots,\l_k$ and $\bl_2,\ldots,\bl_k$,
$${C^*}\models\Phi(\beta_1,\ldots,\beta_l,\l_1,\bl_2,\ldots,\bl_k).$$
Since $j$ is an elementary embedding, 
$${C^*}\models\Phi(\beta_1,\ldots,\beta_l,\bl_1,\bl_2,\ldots,\bl_k).$$
Claim 1 is proved. 

\medskip

\noindent{\bf Claim 2:} If $\l_1<\ldots<\l_k$  are Woodin cardinals, $\kappa_1<\ldots<\kappa_k$ are regular cardinals $>\aleph_1$, $\l_1>\max(\kappa_1,\ldots,\kappa_k)$, and $\beta_1,\ldots,\beta_l<\kappa_1$, then 
$${C^*}\models\Phi(\beta_1,\ldots,\beta_l,\kappa_1,\ldots,\kappa_k)\leftrightarrow
\Phi(\beta_1,\ldots,\beta_l,\l_1,\ldots,\l_k)$$ for all formulas $\Phi(x_1,\ldots,x_l,y_1,\ldots,y_k)$ of set theory. 
\medskip

We use induction on $k$ to prove the claim. The case $k=0$ is clear. Let us assume the claim for $k-1$. Using $\omega$-stationary tower forcing we can find $$j:V\to M, M^\omega\subseteq M, j(\kappa_1)=\l_1, j(\l_i)=\l_i\mbox{ for $i>1$}.$$
Now we use the Claim to prove the theorem. 
Suppose now $${C^*}\models\Phi(\beta_1,\ldots,\beta_l,\kappa_1,\kappa_2,\ldots,\kappa_k).$$ By the induction hypothesis applied to $\kappa_2,\ldots,\kappa_k$ and $\l_2,\ldots,\l_k$,
$${C^*}\models\Phi(\beta_1,\ldots,\beta_l,\kappa_1,\l_2,\ldots,\l_k).$$
Since $j$ is an elementary embedding, 
$${C^*}\models\Phi(\beta_1,\ldots,\beta_l,\l_1,\l_2,\ldots,\l_k).$$
Claim 2 is proved.

The  theorem follows now immediately from Claim 2.
\end{proof}

Note that we cannot extend Theorem~\ref{ind} to $\aleph_1$, for $\aleph_1$ has the following property, recognizable in ${C^*}$, which no other uncountable cardinal has: it is has uncountable cofinality but all of its (limit) elements have countable cofinality.

\begin{theorem}\label{wepe}
If $V=L^\mu$, then ${C^*}$ is exactly the inner model $M_{\omega^2}[E]$, where $M_{\omega^2}$ is the $\omega^2$th iterate of $V$ and $E=\{\kappa_{\omega\cdot n}:n<\omega\}$.
\end{theorem}

\begin{proof}
 In order to prove the theorem, we have to show that in $M_{\omega^2}$ we can recognize which ordinals have cofinality $\omega$ in $V$. The following lemma gives a general analysis about the relation between the cofinality of the ordinal in a universe and its cofinality in an iterated ultrapower of it. 
  
 \begin{lemma} Let $M$ be a transitive model of ZFC+GCH with a measurable cardinal $\kappa$ which is iterable. (Namely the iterated ultrapowers by a normal ultrafilter on $\kappa$ are all well founded.)  For $\beta\in On$ let $M_\beta$ be the $\beta$-th iterate of  $M$. Then for every ordinal $\delta\in M_\beta$ if  $\cof^M(\delta)<\kappa$ then either $\cof^{M_\beta}(\delta)=\cof^M(\delta)$ or there is a limit  $\gamma\leq\beta$ such that $\cof^{M_\beta}(\delta)=\cof^M(\gamma)$.
  \end{lemma}
  
  \begin{proof} Let $j_{\beta,\gamma}$ be the canonical embedding $j_{\beta\gamma}:M_\beta\rightarrow M_\gamma$ and let $\kappa_\beta=j_{0\beta}(\kappa)$. Let $\xi_\beta=(\kappa_\beta^+)^{M_\beta}$ and let $\eta=\cof^M(\xi_0)$.
As in claim 1 of the proof of Theorem 5.5 we can show by induction that $\cof^M(\xi_\beta)=\eta$. 
   \begin{claim} $\cof^M(\kappa_{\beta+1})=\eta$.
   \end{claim}
   \begin{proof}
   $M_{\beta+1}$ is the ultrapower of $M_\beta$ by a normal ultrafilter on $\kappa_\beta$. Since $M_\beta\models GCH$, it is well known that $ \cof^{M_\beta}(j_{\beta,\beta+1}(\kappa_\beta))=\cof^{M_\beta}(\xi_\beta)$ but since we have $\cof^M(\xi_\beta)=\eta $ we get $\cof^M(\kappa_{\beta+1})=\eta$.
   \end{proof}

   Without loss of generality we can assume that $\delta$ (in the formulation of the lemma) is regular in $M_\beta$. We distinguish several cases :
   \begin{description}
     \item[$\delta\leq\kappa$] We know that in this case the iterated ultrapower does not change that cofinality of  $\delta$. Hence $\cof^M(\delta)=\cof^{M_\beta}(\delta)$.
     \item[$\kappa<\delta\leq\kappa_\beta$] An argument like in the proof of claim 2 of the proof of Theorem 5.5 will show that either $\cof^{M}(\delta)=\eta$ or there is $\gamma\leq\beta$ such that $\delta=\kappa_\gamma$. The first case cannot occur since we assumed that $\cof^M(\delta)<\kappa<\eta$. In the second case, if $\gamma$ is successor or $0$ again we get by the previous claim that $\cof^M(\delta)\geq\kappa$,  contradicting again  the assumption. If $\gamma$ is limit, the lemma is verified.
     \item[$\kappa_\beta<\delta$] For simplifying notation let $j=j_{0\beta}$. Every ordinal in $M_\beta$ is of the form $j(F)(\kappa_{\gamma_0}\ldots \kappa_{\gamma_k})$ for some $\gamma_0<\ldots\gamma_k<\beta$ and  $F\in M$ an ordinal valued  function defined on  $\kappa^{k+1}$. In particular for every ordinal $\nu$ in $M_\beta$ there is a function $F\in M$, $F:\kappa^{<\omega}\rightarrow On$, such that $\nu\in j(F)"j(\kappa)^{<\omega}$. Since $\cof^M(\delta)<\kappa$ there is in $M$ an ordinal $\mu<\kappa$ and  a sequence  $\langle F_\eta |\eta<\mu\rangle$ such that for $\eta<\mu$ $F_\eta$ is a function from $\kappa^{<\omega}$ such that the union of the ranges of $\langle j(F_\eta)\cap\delta| \eta<\mu \rangle $ is cofinal in $\delta$ . But $\langle j(F_\eta)\cap\delta| \eta<\mu \rangle=j(\langle F_\eta |\eta<\mu\rangle)\in M_\beta$. But the union of the ranges is the union of $\mu$ sets each of cardinality $\leq j(\kappa)=\kappa_\beta$. In $M_\beta$ $\delta$ is a regular cardinal above $j(\kappa)$, hence this union is bounded in $\delta$. A contradiction.

      \end{description} \end{proof}

   \begin{corollary} If $V\models GCH$, $\kappa$ measurable,  then an ordinal has cofinality $\omega$ in $V$ iff its cofinality in $M_{\omega^2}$ is either $\omega$ or of the form $\kappa_\gamma$ for some limit  $\gamma\leq\omega^2$.
   \end{corollary}
   
   From the point of $M_{\omega^2}$ $E$ is a Prikry generic sequence with respect to the image  of $\mu$.  Hence the only cardinal of $M_{\omega^2}$ that changes its cofinality is $\kappa_{\omega^2}$. So in $M_{\omega^2}$ it is still true that ordinal has cofinality $\omega$ in $V$ iff its cofinality in $M_{\omega^2}[E]$ is in $\{\omega\}\cup E \cup \{\sup(E)\}$. It follows that $C^\ast\subseteq M_{\omega^2}[E]$.
   
   For the other direction, let $\mu'$ be the image of $\mu$ in $M_{\omega^2}$. By Theorem 5.5 we know that the Dodd-Jensen Core model, $K^{DJ}$ is the same as the Dodd-Jensen core model of $C^\ast$. $V=L^\mu$. Hence by claim  2 of the proof of theorem 5.6,  for  $\eta$ regular in $K^{DJ}$,  $\kappa<\eta\leq\kappa_{\omega^2}$, $\cof(\eta)=\omega$ iff $\eta=\kappa_\gamma$ for some $\gamma\leq\omega^2$. By the above lemma we know that for successor $\gamma$ the ordinal $\kappa_\gamma$ has cofinality $\kappa^+$. Hence $E$ is exactly the set of ordinals $\eta$ which are regular in $K^{DJ}$,  $\kappa\leq\eta\leq\kappa_{\omega^2}$ and $\cof(\eta)=\omega$. This shows that $E\in C^\ast$.
   
   It is well known  that if we define the filter $F$ on $\kappa_{\omega^2}$ generated by final segment of $E$ then $M_{\omega^2}=L^{\mu'}=L[F]$ \cite{MR0277346}. Therefore $M_{\omega^2}$ is a definable class in $C^\ast$. We conclude $M_{\omega^2}\subseteq C^\ast$ and hence finally, $M_{\omega^2}[E]\subseteq C^\ast$.
 \end{proof}

The situation is similar with the inner model for two measurable cardinals: To get ${C^*}$ we first iterate the first measurable $\omega^2$ times, then the second $\omega^2$ times, and in the end take two Prikry sequences.

We now prove the important property of $C^*$ that its truth is invariant under (set) forcing. We have to assume large cardinals because conceivably $C^*$ could satisfy $V=L$ but in a (set) forcing extension $C^*$ would violate $V=L$ (see Section~\ref{namba} below).

\begin{theorem}\label{515}
Suppose there is a proper class of Woodin cardinals. Suppose $\oP$ is a forcing notion and $G\subseteq\oP$ is generic. 
Then $$\Th(({C^*})^V)=\Th(({C^*})^{V[G]}).$$ 
Moreover, the theory $\Th({C^*})$ is independent of the cofinality used\footnote{I.e. $\Th(C^*)=\Th(C^*_{<\kappa})$ for all regular $\kappa$.}, and forcing does not change the reals of these models.

\end{theorem}

\begin{proof}
 Let $G$ be $\oP$-generic.
Let us choose a Woodin cardinal $\lambda >|\oP|$. Let $H_1$ be generic for the countable stationary tower forcing $\Q_{<\lambda}$.  In $V[H_1]$ there is a generic embedding  $j_1:V\to M_1$ such that $V[H_1]\models M_1^\omega\subseteq M_1$ and $j(\omega_1)=\lambda$. Hence $({C^*})^{V[H_1]}=({C^*})^{M_1}$   and
$$j_1:({C^*})^{V}\to ({C^*})^{M_1}=({C^*})^{V[H_1]}=(C^*_{<\lambda})^{V}.$$
The last equality uses the fact that an ordinal has cofinality $\omega$ in $V[H_1]$ iff it has cofinality $<\lambda$ in $V$. Now by elementarity $\Th(({C^*})^V)=\Th((C^*_{<\lambda})^{V})$.

Since $|\oP|<\lambda$, $\lambda$ is still Woodin in $V[G]$. Let $H_2$ be generic for the countable stationary tower forcing $\Q_{<\lambda}$ over $V[G]$.   Let $j_2:V[G]\to M_2$ be the generic embedding. Now $V[G,H_2]\models M_2^{\omega}\subseteq M_2$ and $j_2(\omega_1)=\lambda$.
Hence $$j_2:({C^*})^{V[G]}\to ({C^*})^{M_2}=({C^*})^{V[G,H_2]}=(C^*_{<\lambda})^{V[G]}=(C^*_{<\lambda})^{V}.$$
and therefore by elementarity $({C^*})^V\equiv(C^*_{<\lambda})^V\equiv({C^*})^{V[G]}$.

   We know  (Theorem~\ref{mahlo}) that under the existence of a Woodin cardinal the set of reals of $C^\ast$ is countable (in $V$). Hence when we define $j_1:V\rightarrow M_1$ there are no new reals added to the $C^\ast$ of the corresponding models. Hence the reals of $(C^\ast)^V$ are the same as the reals of $(C^\ast)^{M_1}$. We argued that the last model is exactly $(C^\ast_{<\lambda})^V$. The same is true in $V[G]$. But $C^\ast_{<\lambda}$ does not change when we move from $V$ to $V[G]$. So $(C^\ast)^V$ and $(C^\ast)^{V[G]}$ have the same reals.
      
     The argument for the elementary equivalence of $C^\ast$ and $C^\ast_{<\kappa}$ for $\kappa$ regular, proceeds in a similar     manner. We use  stationary tower forcing which produces an elementary embedding $j$ with critical point $\kappa$ such that $j(\kappa)=\lambda$, where $\lambda$ is a Woodin cardinal above $\kappa$. Then we argue that $j(C^\ast_{<\kappa})$ is $(C^\ast_{<\lambda})^V$.
  
\end{proof}

We may ask, for which $\lambda$ and $\mu$ is $\cofmodel{<\l}=\cofmodel{<\mu}$?
Observations:

\begin{itemize}


\item It is possible that $C^*\ne\cofmodel{<\omega_2}$. Let us use the $\le\omega$-closed stationary tower  forcing of \cite[Section 1]{MR1359154} to map $\omega_3$ to $\lambda$. In this model $V_1$ the inner model ${C^*}$ is preserved. It is easy to see  that in the extension the set $A$ of ordinals  below $\l$ of cofinality $\omega_1$ is not in $V$. If $C^*=\cofmodel{<\omega_2}$, then $A$ is in $\cofmodel{<\omega_2}$. We are done.

\item It is possible that $C^*$ changes. Extend the previous model $V_1$ to $V_2$ by collapsing $\omega_1$ to $\omega$. Then $(C^*)^{V_2}=(\cofmodel{<\omega_2})^{V_1}\ne (C^*)^{V_1}$. So $C^*$ has changed.

\item Question: Does a Woodin cardinal imply $\cofmodel{<\omega_2}\ne\cofmodel{\omega,\omega_1}$?

\end{itemize}


We do not know whether the CH is true or false in ${C^*}$. Forcing absoluteness of the theory of ${C^*}$ under the hypothesis of large cardinals implies, however, that large cardinals decide the CH in $C^*$ in forcing extensions. This would seem to give strong encouragement to try to solve the problem of CH in $C^*$. The situation is in sharp contrast to $V$ itself where we know that large cardinals definitely do not decide CH \cite{MR0224458}. 
We can at the moment only prove that the size of the continuum of $C^*$ is at most $\omega_2^V$. In the presence of a  Woodin cardinal this tells us absolutely nothing, as then $\omega_2^V$ is (strongly) Mahlo in ${C^*}$ (Theorem~\ref{mahlo}), and hence certainly far above the continuum of $C^*$. So the below result is mainly interesting because it is a provable result of ZFC, independent of whether we assume the existence of Woodin cardinals. 
However, we show later that in the presence of large cardinals 
there is a cone of reals $x$ such that the relativized version of ${C^*}$, 
$C^*\hspace{-.5mm}(x)$, satisfies $CH$. In the light of this it is tempting to conjecture that CH is indeed true in $C^*$, assuming again the existence of sufficiently large cardinals.

\begin{theorem}\label{5930a}
$|\P(\omega)\cap{C^*}|\le \aleph_2.$
\end{theorem}

\begin{proof} We use the notation of Definition~\ref{defin}. 
Suppose $a\subseteq\omega$ and $a\in L'_\xi$ for some $\xi$. Let $\mu>\xi$ be a sufficiently large cardinal. We  build an  increasing elementary chain $(M_\alpha)_{\alpha<\omega_1}$ such that 
\begin{enumerate}
\item $a\in M_0$ and $M_0\models a\in C^*$.
\item $|M_\alpha|\le\omega$.
\item $M_\alpha\prec H(\mu)$.

\item $M_\gamma=\bigcup_{\alpha<\gamma}M_\alpha$, if $\gamma=\cup\gamma$.

\item If $\beta\in M_\alpha$ and $\cof^V(\beta)=\omega$, then $M_{\alpha+1}$ contains 
 an $\omega$-sequence from $H(\mu)$, cofinal in $\beta$.
\item If $\beta\in M_\alpha$ and $\cof^V(\beta)>\omega$ then for unboundedly many $\gamma<\omega_1$ there is $\rho\in M_{\gamma+1}$ with
$$\sup(\bigcup_{\xi<\gamma}(M_\xi\cap\beta))<\rho<\beta.$$
\end{enumerate}
Let $M$ be $\bigcup_{\alpha<\omega_1}M_\alpha$,  $N$  the transitive collapse of $M$, and  $\zeta$ the ordinal $N\cap On$. Note that $|N|\le\omega_1$, whence $\zeta<\omega_2$. By construction, an ordinal in $N$ has  cofinality $\omega$ in $V$ if and only if it  has cofinality $\omega$ in $N$. Thus $(L'_\xi)^N=L'_\xi$ for all $\xi<\zeta$.  Since $N\models a\in C^*$, we have $a\in L'_\zeta$. The claim follows.
\end{proof}

The proof of Theorem~\ref{5930a} gives the following more general result:

\begin{theorem}\label{5930}Let $\kappa$ be a regular cardinal and $\delta$ an ordinal. Then
$$|\P(\delta)\cap{\cofmodel{\kappa}}|\le (|\delta|\cdot\kappa^+)^{+}.$$
\end{theorem}
\begin{corollary}\label{lsthm} If $\delta\ge\kappa^+$ is a cardinal in $\cofmodel{\kappa}$ and $\l=|\delta|^+$, then $\cofmodel{\kappa}\models 2^\delta\le\l$. 
\end{corollary}

\begin{corollary}
Suppose $V=C^*$. Then $2^{\aleph_\a}=\aleph_{\a+1}$ for $\a\ge 1$, and $2^{\aleph_0}=\aleph_1$ or  $2^{\aleph_0}=\aleph_2$.
\end{corollary}

\begin{theorem}\label{5930b}
Suppose $E=\{\alpha<\omega_2^V:\cof^V(\alpha)=\omega_1^V\}$. Then
$\diamondsuit_{\aleph^V_2}(E)$ holds in ${C^*}$.
\end{theorem}


\begin{proof}The proof is as the standard proof of $\diamondsuit_{\aleph_2}(E)$ in $L$, with a small necessary patch. 
We  construct a sequence $s=\{(S_\alpha, D_\alpha):\alpha<\aleph_2^V\}$ taking always for limit $\alpha$ the pair $(S_\alpha, D_\alpha)$ to be the least $(S, D)\in L'_{\aleph^V_2}$  in the well-order (see Proposition~\ref{wpoa}) 
$$R=\{(a,b)\in (L'_{\aleph^V_2})^2\ : 
\ L'_{\aleph^V_2}\models \Psi_{\L(Q^{\cof}_{\omega})}(a,b)\}$$ 
such that  $S\subseteq\a$, $D\subseteq\a$ a club, and $S\cap\beta\ne S_\beta$ for $\beta\in D$, if any exists, and $S_\a=D_\a=\a$ otherwise. 
Note that $s\in C^*$. We show that the sequence $s$ is a diamond sequence in $C^*$. Suppose it is not and $(S,D)\in C^*$ is a counter-example, $S\subseteq\aleph_2^V$ and $D\subseteq\aleph_2^V$ club such that $S\cap\beta\ne S_\beta$ for all $\beta\in D$. As in the proof of Theorem~\ref{5930a} we can construct $M\prec H(\mu)$ such that $|M|=\aleph^V_1$, the order-type of $M\cap \aleph_2^V$ is in $E$, $\{s,(S,D)\}\subset M$, and if $N$ is the transitive collapse of $M$, with ordinal $\delta\in E$, then $\{s\restriction\delta,(S\cap\delta,D\cap\delta)\}\subset N$
and $(L'_\xi)^N=L'_\xi$ for all $\xi<\delta$. Because of the way $M$ is constructed, the well-order $R$ restricted to $L'_\delta$ is defined in $M$ on $L'_{\aleph^V_2}$ by the same formula 
${\Psi}_{\L(Q^{\cf}_\kappa)}(x,y)$ as $R$ is defined on $L'_{\aleph^V_2}$ in $H(\mu)$.
Since $S\cap\delta\in L'_{\aleph^V_2}$ and $S\cap\beta\ne S_\beta$ for $\beta\in D\cap\delta$, we may  assume, w.l.o.g., that
$(S,D)\in L'_{\aleph^V_2}$. Furthermore, we may assume, w.l.o.g.,  that $(S,D)$ is the $R$-least counter-example to $s$ being a diamond sequence. Thus the pair $(S\cap\delta,D\cap\delta)$ is the $R$-least $(S', D')$ such that  $S'\subseteq\delta$, $D'\subseteq\delta$ a club, and $S'\cap\beta\ne S'_\beta$ for $\beta\in D'$. It follows that  $(S',D')=(S_\delta,D_\delta)$ and, since $\delta\in D$, a contradiction.

\end{proof}

A problem in using condensation type arguments, such as we used in the proofs of Theorem~\ref{5930a} and Theorem~\ref{5930b} above, is the non-absoluteness of 
${C^*}$. There is no reason to believe that $({C^*})^{{C^*}}=C^*$ in general
 (see Theorem~\ref{2479}). Moreover, we prove in Theorem~\ref{chfalse} the consistency of $C^*$ failing to satisfy CH, relative to the consistency of an inaccessible cardinal.

We now prove that CH holds in $\cy$ for a cone of $y$. By $\cy$ we mean the extension of $C^*$ in which the real $y$ is allowed as a parameter throughout the construction. 

Suppose $N$ is a well-founded model of ${\zfc}^-$ and $N$ thinks that $\l\in M$ is a Woodin cardinal. We say that $N$ is {\bf iterable}, if  all countable iterations of forming  generic ultrapowers of $N$ by  stationary tower forcing at $\l$ are well-founded. If $N\prec H(\theta)$  for large enough $\theta$ and $N$ contains a  measurable cardinal  (of $V$) above the Woodin cardinal, then it is iterable for the following reason: Suppose $\alpha<\omega_1$.
Suppose $N^*$ is an iteration of $N$ at the measurable cardinal until $\alpha\le N^*\cap \On$.
This is well-founded because it can be embedded into a long enough iteration of $V$ at the measurable cardinal.   It is well known that  an iteration of length $\alpha\le N^*\cap \On$ of forming generic ultrapowers of $N^*$  by stationary tower forcing  is  well founded (see e.g. \cite[Lemma 4.5]{MR2723878}). Now the iteration of forming generic ultrapowers of $N$ by stationary tower forcing can be embedded to the corresponding iteration of $N^*$. Since the latter iteration is well-founded, so is the former.

%
We use the notation $L'_\a(y)$ for the levels of the construction of ${\cy}$.

\begin{lemma}\label{sublemma}Suppose there is  Woodin cardinal and a measurable cardinal above it.
Suppose $a\subseteq\gamma<\omega_1^V$.
Then the following conditions are equivalent:

\begin{description}
\item[(i)]  $a\in C^*\hspace{-.5mm}(y)$. 

\item[(ii)]  There is a countable transitive iterable model $N$ of $ZFC^-$+``{\it there is a Woodin cardinal}" such that  $\{a,y\}\subset N$, $\gamma<\omega_1^N$, and $N\models ``a\in {\cy}"$.
\end{description}

\end{lemma}

\begin{proof}
(i)$\to$(ii): Suppose first $a\in\cy$. 
 Pick a large enough $\theta$ and a countable $M\prec H_\theta$ such that $\gamma\cup\{\gamma,a,y\}\subset M$ and both the Woodin cardinal and the measurable above it are in  $M$. Then $M$ is iterable. Let $\pi:M\cong N$ with $N$ transitive. This $N$ is as required in (ii). In particular, $\pi(a)=a, \pi(y)=y$ and $N\models ``a\in {\cy}"$ since $M\prec H_\theta$.

(ii)$\to$(i): Suppose $N$ is as in (ii). Since $N\models ``a\in\cy"$, there is $\bar{\beta}<\omega_2^N$ such that $N\models a\in L'_{\bar{\beta}}(y)$. 
We  form an iteration sequence $\{N_\a:\a<\omega^V_1\}$ with elementary embeddings $\{\pi_{\a\beta}:\a<\beta<\omega_1\}$. Let $N_0=N$.   Let $N_{\g+1}$ be the transitive collapse of a generic ultrapower of the stationary tower on the image of $\lambda$ in $N_\g$. Let $\pi_{\a{\a+1}}$ be the canonical embedding $N_\a\to N_{\a+1}$.
For limit $\a\le\omega_1^V$, the model $N_\a$ is the transitive collapse of the direct limit of  the models $N_\beta$, $\beta<\a$, under the mappings $\pi_{\beta\g}$, $\beta<\g<\a$. By the iterability condition each $N_\a$ is well-founded, so the transitive collapse exists. 
Since $\pi_{0\g}(\omega_1^N)$ is extended in each step of this iteration of length $\omega_1$ of countable models, $\pi_{0\omega_1}(\omega_1^N)=\omega_1^V$. 
Moreover, 
$\pi_{0\omega_1}(y)=y$ and $\pi_{0\omega_1}(a)=a$, as $a\subseteq\gamma$ and $\gamma<\omega_1^N$. Now
by elementarity, $$N_{\omega_1}\models ``a\in L'_{\beta}(y)",$$ where
$\beta=\pi_{0\omega_1}(\bar{\beta})$.\medskip
 We now show \begin{equation}\label{sama}
(L'_\beta(y))^{N_{\omega_1}}=(L'_\beta(y))^{V}.
\end{equation}
This is proved level by level. If $N_{\omega_1}\models``\cof(\delta)=\omega"$, then of course $\cof(\delta)=\omega$. Suppose then $N_{\omega_1}\models``\cof(\delta)>\omega"$, where $\delta<\beta$. Let $\bar{\delta}<\bar{\beta}$ such that $\pi_{0\omega_1}(\bar{\delta})=\delta$. Then  $\bar{\delta}<\omega_2^N$. Thus $N\models\cof(\bar{\delta})=\omega_1$. By elementarity, $N_{\omega_1}\models\cof(\delta)=\omega_1$. But $\omega_1^{N_{\omega_1}}=\omega_1^V$. Hence $\cof(\delta)=\omega_1^V$. Equation~(\ref{sama}) is proved.

Now we can prove (i): Since $a\in (L'_\beta(y))^{N_{\omega_1}}$,  equation~(\ref{sama}) implies $a\in L'_\beta(y)\subset \cy$.  

\end{proof}

Note that condition (ii) above is a $\Sigma^1_3$-condition. Thus, if there is a Woodin cardinal and a measurable above, then the set of reals of $C^*$ is a countable $\Sigma^1_3$-set with a $\Sigma^1_3$-well-ordering.

\begin{lemma}\label{epppow}Suppose there is a Woodin cardinal and a measurable cardinal above it.
Then the following conditions are equivalent:

\begin{description}
\item[(i)]  $C^*\hspace{-.5mm}(y)\models CH$. 

\item[(ii)]  There is a countable transitive iterable model $M$ of $ZFC^-$ plus ``{\it there is a Woodin cardinal}" such that $y\in M$, $M\models ``{\cy}\models CH"$, and $(\P(\omega)\cap \cy)^M=\P(\omega)\cap \cy $.
\end{description}
\end{lemma}

\begin{proof}
(i)$\to$(ii): Since we assume the existence of a Woodin cardinal, there are only countably many reals in ${\cy}$. Let $\gamma=\omega_1^{{\cy}}$. Thus $\gamma<\omega_1^V$. By (i) we may find a subset $a\in\cy$ of  $\gamma$ that codes an enumeration of $\P(\omega)\cap {\cy}$ in order-type $\gamma$ together with a well-ordering of $\omega$ of each order-type $<\gamma$. By lemma~\ref{sublemma} there is a countable transitive iterable model $N$ of $ZFC^-$+``{\it there is a Woodin cardinal}" such that  $\{a,y\}\subset N$, $\gamma<\omega_1^N$ and $N\models ``a\in {\cy}"$. We show $N\models``{\cy}\models CH"$. Suppose $b\in N$ is real such that $N\models``b\in \cy"$. By Lemma~\ref{sublemma}, $b\in\cy$. Hence $b$ is coded by $a$. The length of the sequence $a$ is $\gamma$, so we only have to show that $N\models ``\gamma\le\omega_1^{{\cy}}"$. Suppose $N\models``\gamma > \omega_1^{{\cy}}"$. In such a case, by assumption, $a$ codes a well-ordering $R$ of $\omega$ of order-type $(\omega_1^{\cy})^N$. But $N\models ``a\in\cy"$, whence $N\models``R\in\cy"$, a contradiction. 
The proof that
$(\P(\omega)\cap \cy)^N=\P(\omega)\cap \cy$ is similar.

Assume then (ii). Let $N$ be as in (ii).  Let $\gamma=(\omega_1^{\cy})^N$. Since $N\models``\cy\models CH"$, we can let some $a\subseteq\gamma$ code $(\P(\omega)\cap\cy)^N$ as a sequence of order-type $\gamma$. By lemma~\ref{sublemma}, $a\in\cy$. Since $(\P(\omega)\cap \cy)^N=\P(\omega)\cap \cy$,  $a$ is an enumeration of all the reals in $\cy$ and $\gamma=\omega_1^{\cy}$. 
Hence $\cy\models CH$.
\end{proof}

Note that condition (ii) above is a $\Sigma^1_4$-condition. Also, forgetting $y$, $``C^*\models CH"$ itself is a $\Sigma^1_4$-sentence of set theory.

\def\cx{{C^*\hspace{-.5mm}(x)}}


\medskip

Using the above two Lemmas, we now prove a result which seems to lend support to the idea that $C^*$ satisfies CH, at least assuming large cardinals. Let $\le_T$ be the Turing-reducibility relation between reals.
The {\em cone} of a real $x$ is the set of all reals $y$ with $x\le_T y$. A set of reals is called a {\em cone} if it is the cone of some real. Suppose $A$ is a projective set of reals closed under Turing-equivalence. If we assume PD, then by  a result of D. Martin \cite{MR0227022} there is a cone which is included in $A$ or is disjoint from $A$. 

\begin{theorem}\label{coneresult}
If there are three  Woodin cardinals and a measurable cardinal above them, then there is a cone of reals $x$ such that $C^*\hspace{-.5mm}(x)$ satisfies the Continuum Hypothesis.
\end{theorem}

\begin{proof}
We first observe that if two reals $x$ and $y$ are Turing-equivalent, then $C^*\hspace{-.5mm}(x)=C^*\hspace{-.5mm}(y)$. Hence the set 
$$A=\{y\subseteq \omega:C^*\hspace{-.5mm}(y)\models CH\}$$
 is closed under Turing-equivalence, and therefore by \cite{MR955605} amenable to the above mentioned result by Martin on cones. We already know from Lemma~\ref{epppow} that the set $A$ is projective, in fact $\Sigma^1_4$. Now we need to show
that for every real $x$ there is a real $y$ such that $x\le_T y$ and $y$ is in the set.
Fix $x$. Let $\oP$ be the standard forcing which, in $\cx$, forces a subset $B$ of $\omega_1^\cx$, such that $B$ codes, via the canonical pairing function in $\cx$, an onto mapping $\omega_1^\cx \to \P(\omega)\cap\cx$.  Let $B$ be $\oP$-generic over $\cx$. Note that $\oP$ does not add any new reals. Now we code $B$ by a real by means of almost disjoint forcing. Let $Z_\a$, $\a<\omega_1^\cx$, be a sequence in $\cx$ of almost disjoint subsets of $\omega$. Let $\oQ$ be the standard CCC-forcing, known from \cite{MR0270904}, for adding a real $y'$ such that for all  $\a<\omega_1^\cx$:
$$|z_\a\cap y'|\ge\omega\iff \a\in B.$$ Let $y=x\oplus y'$. Of course, $x\le_T y$. Now $$\cx\subseteq \cx[B]\subseteq C^*\hspace{-.5mm}(y).$$ By the definition of $B$, $\cx[B]\models CH$. The forcing $\oQ$ is of cardinality $\aleph_1$ in $\cx[B]$, hence $C^*\hspace{-.5mm}(y)\models CH$.
\end{proof}

Assuming large cardinals, the set of reals of $C^*$ seems like an interesting countable $\Sigma^1_3$-set with a $\Sigma^1_3$-well-ordering. It might be interesting to have a better understanding of this set. This set is contained in the reals of the so called $M^\sharp_1$, the smallest inner model for a Woodin cardinal (M. Magidor and R. Schindler, unpublished).

In Part 2 of this paper we will consider the so-called {\em stationary logic} \cite{MR486629}, a strengthening $\L(\aaq)$ of $\L(Q^{\cof}_{\omega})$, and the arising inner model $C(\aaq)$, a supermodel of $C^*$. We will show that, assuming a proper class of measurable Woodin cardinals, uncountable regular cardinals are measurable in $C(\aaq)$, and the theory of $C(\aaq)$ is absolute under set forcing. These results remain true if we enhance the expressive power of $\L(\aaq)$ slightly, and then the inner model arising from the enhanced stationary logic satisfies the Continuum Hypothesis, assuming again a proper class of measurable Woodin cardinals.  

\section{Consistency results about $C^*$}\label{namba}

We define a version of Namba forcing that we call {\em modified Namba forcing} and then use this to prove consistency results about $C^*$.

 Suppose $S=\{\l_n : n<\omega\}$ is a  sequence of regular cardinals $>\omega_1$ such that every $\l_n$ occurs infinitely many times in the sequence. 
Let $\la B_n : n<\om\ra$ be a partition of $\omega$ into infinite sets. 

\begin{definition}\label{strees}
The forcing $\P$ is defined as follows: Conditions are  trees $T$ with $\omega$ levels, consisting of finite sequences of ordinals, defined as follows: If $(\a_0,\ldots,\a_{i})\in T$, let 
$$\suc_T((\a_0,\ldots,\a_{i}))=\{\beta: (\a_0,\ldots,\a_{i},\beta)\in T\}.$$ 
The forcing $\P$ consists of trees, called $S$-trees, such that if $(\a_0,\ldots,\a_i)\in T$
and $i\in B_n$, then
\begin{enumerate}
 
\item $|\suc_T((\a_0,\ldots,\a_{i-1}))|\in\{1,\l_n\}$, 

\item For every $n$ there are $\a_i,\ldots, \a_k$ such that $k\in B_n$ and $|\suc_T((\a_0,\ldots,\a_{k}))|=\l_n$.

\end{enumerate}

\noindent If $|\suc_T((\a_0,\ldots,\a_{i-1}))|=\l_n$, we call $(\a_0,\ldots,\a_{i-1})$ a {\em splitting point} of $T$. Otherwise $(\a_0,\ldots,\a_{i-1})$ is a {\em non-splitting point} of $T$. The stem $\stem(T)$ of $T$ is the maximal (finite) initial segment that consists of non-splitting points. If $s=(\a_0,\ldots,\a_i)\in T$, then $$T_s=\{(\a_0,\ldots,\a_i,\a_{i+1},\ldots,\a_n) \in T: i\le n<\omega\}.$$
 A condition $T'$ {\em extends} another condition $T$, $T'\le T$, if $T'\subseteq T$. If $\la T_n : n<\omega \ra\}$ is a generic sequence of conditions, then the stems of the trees $T_n$ form a sequence $\la\a_n : n<\omega\ra$ such that $\la\a_i: i\in B_n\ra$ is cofinal in $\l_n$. Thus in the generic extension $\cof(\l_n)=\omega$ for all $n<\omega$. 

\end{definition} 
 
We shall now prove that no other regular cardinals get cofinality $\omega$.

\begin{proposition}Suppose $\kappa\notin S\cup\{\omega\}$ is regular. Then $\P\force\cof(\kappa)\ne\omega$.
\end{proposition}
 
\begin{proof}
Let us first prove that if $\tau$ is a name for an  ordinal, then for all $T\in\P$ there is $T^*\le T$ such that $\stem(T^*)=\stem(T)$ and if  $T^{**}\le T^*$ decides which ordinal $\tau$ is, and $s=\stem(T^{**})$, then $T^*_{s}$ decides $\tau$. Suppose $T$ is given and the length of its stem is $l\in B_n$. Let us look at the level $l+1$ of $T$.  
Let us call a node $s$ on level $l+1$ of $T$ {\em good} if the claim is true when $T$ is taken to be $T_s$. Suppose first there are $\l_n$ good nodes. For each good $s$ we choose $T^*(s)\le T_s$ such that $\stem(T^*(s))=\stem(T_s)$ and if  some $T^{**}\le T^*(s)$ decides which ordinal $\tau$ is, and $s'=\stem(T^{**})$, then already $T^*(s)_{s'}$ decides $\tau$. W.l.o.g. the length of such $s'$ is a fixed $k$. 
We get the desired $T^*$ by taking the fusion.
 Suppose then there are not $\l_n$ many good nodes. So there must be $\l_n$ bad nodes. We repeat this process on the next level. Suppose the process does not end. We get  $T'\le T$ consisting of bad nodes.  Since $T$ forces that $\tau$ is an ordinal, there is $T''\le T'$ such that $T''$ decides which ordinal $\tau$ is. We get a contradiction: the node of the stem of $T''$, which is also a node of $T'$, cannot be a bad one.

Suppose now $\la\beta_n : n<\omega\ra$ is a name for an  $\omega$-sequence of ordinals below $\kappa$,  and $T\in\P$ forces this.
We construct $T^*\le T$ and an ordinal $\delta<\kappa$ such that $T^*$ forces the sequence $\la\beta_n : n<\omega\ra$ to be bounded below $\kappa$ by $\delta$.
For each $n$ we have a partial function $f_n$ defined on  $s\in T$ of  such that if $T_s$ decides a value for $\beta_n$ and then the value is defined to be $f_n(s)$. Let us call $T$ {\em good for
 $\la\beta_n : n<\omega\ra$} if for all infinite  branches $B$
through $T$ and all $n$ there is $k$ such that $f_n$ restricted to the initial segment of $B$ of length $k$ is defined. It follows from the above  that we can build, step by step a $T^*\le T$ with the same stem as $T$ such that $T^*$ is good for  $\la\beta_n : n<\omega\ra$.

Without loss of generality,  $T$ itself is good for $\la\beta_n:n<\omega\ra$. Fix $\delta<\kappa$. We consider the following game $G_\delta$. During the game the players determine an infinite branch through $T$. If the game has reached node $t$ on height $k$ with $k+1\in B_n$ we consider two cases:
\smallskip

\noindent Case 1: $\kappa>\l_n$. Bad moves by giving an immediate successor of $t$.
\smallskip

\noindent Case 2: $\kappa <\l_n$. First Bad plays a subset $A$ of (not necessarily immediate) successors  of $t$ such that $|A|<\l_n$. Then Good moves a successor not in the set. 

\smallskip
\noindent Good player loses this game if at some stage of the game a member of the sequence $\beta_n$ is forced to go above $\delta$. Note that the game is determined.
\medskip

\noindent{\bf Main Claim: }There is $\delta<\kappa$ such that Bad does not win $G_\delta$ (hence Good wins).

\begin{proof}
\noindent Assume the contrary, i.e. that Bad wins for all $\delta<\kappa$. Let $\tau_\delta$ be a strategy for Bad for any given $\delta<\kappa$. 
Let $\theta$ be a large enough cardinal and $M\prec H_\theta$ such that $T,\P,\{\beta_n:n<\omega\},\{\lambda_n:n<\omega\}, \{(\delta,\tau_\delta):\delta<\kappa\}$ etc are in $M$, $\alpha\subseteq M$ whenever $\alpha\in M\cap\kappa$ and $|M|<\kappa$. Let $\delta=M\cap\kappa$. We define a play of $G_\delta$ where Bad uses $\tau_\delta$ but all the individual moves are in $M$. Suppose we have reached a node $t$ of $T$ such that $\len(t)=k$ and $k+1\in B_n$. If $\kappa>\lambda_n$, $\lambda_n\subseteq M$, so the move of Bad is in $M$. Suppose then $\kappa<\lambda_n$. The strategy $\tau_\delta$ tells Bad to play a set $A$ of successors (not necessarily immediate) of $t$ such that $|A|<\lambda_n$. The next move of Good has to avoid this set $A$. Still we want the move of Good to be in $M$. We look at all the possibilities according to all the strategies $\tau_\nu$, $\nu<\kappa$. If the play according to $\tau_\nu$ has reached $t$ the strategy $\tau_\nu$ gives a set $A_\nu$ of size $<\lambda_n$ of  successors (not necessarily immediate) of $t$. Let $B$ be the union of all these sets. Still $|B|<\lambda_n$, as $\lambda_n$ is regular. By elementarity, $B\in M$, hence Good can play a successor of $t$ staying in $M$. Since Bad is playing the winning strategy $\tau_\delta$, he should win this play. However, Good can play all the moves inside $M$ without losing. This is a contradiction. \end{proof}

Now we return to the main part of the proof. By the Main Claim there is $\delta$ such that Good wins $G_\delta$.
Let us look at the subtree of all plays of $G_\delta$ where Good plays her winning strategy. A subtree $T^*$ of $T$ is generated and $T^*$ forces the sequence  $\la\beta_n:n<\omega\ra$  to be bounded by $\delta$. \end{proof}

The above modified Namba forcing permits us to carry out the following basic construction: Suppose $V=L$. Let us add a Cohen real $r$. 
We can code this real with the above modified Namba forcing so that in the end for all  $n<\omega$:
$${\cof}^V(\aleph^L_{n+2})=\omega\iff n\in r.$$ Thus in the extension $r\in C^*$.

\begin{theorem}\label{2479}
$Con(ZF)$ implies $Con({(C^*)}^{{C^*}})\ne{C^*})$.
\end{theorem}

\begin{proof}
We start with $V=L$. We add a Cohen real $a$. In the extension ${C^*}=L$, for cofinalities have not changed, so to decide whether $\cof(\a)=\omega$ or not it suffices to decide this in $L$. With  modified Namba forcing we can change---as above---the cofinality of $\aleph^L_{n+2}$ to $\omega$ according to whether $n\in a$ or $n\notin a$. In the extension ${C^*}=L(a)$, for cofinality  $\omega$ has only changed from $L$  to the extent that the cofinalities of $\aleph^L_n$ may have changed, but this we know by looking at $a$. Thus ${(C^*)}^{{C^*}}={(C^*)}^{L(a)}=L$, while $C^*\ne L$. Thus ${(C^*)}^{{C^*}}\ne{C^*}$.
\end{proof}

 We now prepare ourselves to iterating this construction in order to code more sets into $C^*$.

\begin{definition}[Shelah]Suppose $S=\{\l_n : n<\omega\}$ is a  sequence of regular cardinals $>\omega_1$.
 A forcing notion $\P$ satisfies the {\em $S$-condition} if player II has a super strategy (defined below) in the following came $G$ in which the players contribute a tree of finite sequences of ordinals:
 
\begin{enumerate}
\item There are two players I and II and $\omega$ moves.
\item In the start of the game player I plays a tree $T_0$ of finite height  and a function $f:T_0\to\P$ such that for all $t,t'\in T_0$: $t<_{T_0} t'\Rightarrow f(t')<_{\P}f(t)$.
\item Then II decides what the successors of the top nodes of $T_0$ are and extends $f$.   
\item Player I extends the tree with non-splitting nodes of finite height and extends $f$.
\item Then II decides what the successors of the top nodes are and extends $f$.
\item etc, etc
\end{enumerate}
Player II wins if the resulting tree $T$ is an $S$-tree (see Definition~\ref{strees}), and for every $S$-subtree $T^*$ of $T$ there is a condition $B^*\in \P$ such that 
$$B^*\force``\mbox{The $f$-image of some branch through $T^*$ is included in the generic set}".$$
A {\em super strategy} of II is a winning strategy in which the moves depend only on the predecessors in $T$ of the current node, as well as on their $f$-images.\end{definition}

By \cite[Theorem 3.6]{MR1623206} (see also \cite[2.1]{MR2811288}), 
revised countable support iteration of forcing with the $S$-condition does not collapse $\aleph_1$. 


\begin{lemma}
 Modified Namba forcing satisfies the $S$-condition.
\end{lemma}

\begin{proof}
Suppose the game has progressed to the following:

\begin{enumerate}
\item A tree $T$ has been constructed, as well as $f:T\to\P$.
\item Player I has played a non-splitting end-extension $T'$ of $T$. 
\end{enumerate}
 
\noindent Suppose $\eta$ is a maximal node in $T'$. We are in stage $n$. Now II adds $\l_n$ extensions to $\eta$. Let $E$ denote these extensions. Let $B$ be the $S$-tree $f(\eta)$. Find a node $\rho$ in $B$ which is a splitting node and splits into $\l_n$ nodes. Let $g$ map the elements of $E$ 1-1 to successors of $\rho$ in $B$. Now we extend $f$ to $E$ by letting the image of $e\in E$ be the subtree $B_\eta$ of $B$ consisting of $\rho$ and the  predecessors of $\rho$ extended by first $g(e)$ and then the subtree of $B$ above $g(e)$. 

We can easily show that  this is a super strategy. We show that II wins. Suppose $T$ is a tree resulting from II playing the above strategy. Let $T^*$ be any $S$-subtree of $T$. We construct an $S$-tree $B^*\in\P$ as follows. Let $B^*$ be the union of all the stems of the trees $B_\eta$, where $\eta$ is a splitting point of $T^*$. Clearly, $B^*$ is an $S$-tree. To see that  $$B^*\force``\mbox{The $f$-image of some branch through $T^*$ is included in the generic set}",$$ let $G$ be a generic containing $B^*$. This generic is a branch $\g$ through $B^*$. In view of the definition of $B^*$, there is a branch $\beta$ through $T^*$ such that $f``\beta=\g$.

\end{proof}

\begin{theorem}\label{wrep}
Suppose $V=L$ and $\kappa$ is a cardinal of cofinality $>\omega$. There is a forcing notion $\oP$ which forces $C^*\models 2^\omega=\kappa$ and preserves cardinals between $L$ and $C^*$. 
\end{theorem}

\begin{proof}Suppose $V=L$. Let us add $\kappa$ Cohen reals $\{r_\a:\a<\kappa\}$. We code these reals with revised countable support (see \cite{MR1623206}) iterated modified Namba forcing so that in the end we have a forcing extension in which for $\a<\kappa$ and $n<\omega$:
$${\cof}^V(\aleph^L_{\omega\cdot\alpha+n+2})=\omega\iff n\in r_\a.$$ Thus in the extension $r_\a\in C^*$ for all $\a<\kappa$.
We can now note that in the extension $C^*=L[\{r_\a:\a<\kappa\}]$.
First of all, each $r_\a$ is in $C^*$. This gives $``\supseteq"$. For the other direction, we note that whether an ordinal has cofinality $\omega$ in $V$ can be completely computed from the set $\{r_\a:\a<\kappa\}$.  
\end{proof}

Note that the above theorem gives a model in which, e.g. $C^*\models 2^\omega=\aleph_3$, but then in the extension $|\aleph_3^{C^*}|=\aleph_1$, so certainly $V\ne C^*$. Note also, that the above theorem starts with $V=L$, so whether large cardinals, beyond those consistent with $V=L$, decide CH in $C^*$, remains open.

\begin{theorem}\label{chfalse}The following conditions are equivalent:

\begin{description}
\item[(i)] ZF+``there is an inaccessible cardinal" is consistent. 
\item[(ii)] ZFC+``$V=C^*$ and  $2^{\aleph_0}=\aleph_2$" is consistent.

\end{description}
\end{theorem}

\begin{proof}(i)$\to$(ii):
We start with an inaccessible $\kappa$ and $V=L$. We iterate over $\kappa$ with revised countable support forcing adding Cohen reals and coding generic sets using modified Namba forcing. Suppose we are at a stage $\alpha$ and we need to code a real $r$. We choose $\omega$ uncountable cardinals below $\kappa$ and code the real $r$ by changing the cofinality of some of these cardinals to $\omega$. We do this only if at stage $\a$ we already have enough reals in order to code the new $\omega$-sequences by reals. In the end all the reals are coded by changing cofinalities to $\omega$, and at the same time the $\omega$-sequences witnessing the cofinalities are coded by reals. In consequence we have in the end $V=C^*$. The iteration satisfies the $S$-condition, hence $\aleph_1$ is preserved, but the cardinals used for coding the reals all collapse to $\aleph_1$. Hence $\kappa$ is the new $\aleph_2$. In the extension $2^{\aleph_0}=\aleph_2$ and $V=C^*$.

\noindent(ii)$\to$(i): Suppose $V=C^*$ and $C^*\models``2^\omega\ge\omega_2^V"$ but $\omega_2^V$ is not inaccessible in $L$. Then $\omega^V_2=(\lambda^+)^L$ for some $L$-cardinal $\lambda$. Let $A\subseteq\omega_1^V$ such that $A$ codes the countability of all ordinals $<\omega_1^V$ and also codes a well-ordering of $\omega_1$ of order-type $\lambda$. Now $L[A]\models\omega_1^V=\omega_1\wedge\omega_2^V=\omega_2$. We show now that $(C^*)^{L[A]}=C^*$. For this to hold it suffices to show that $L[A]$ agrees with $V$ about cofinality $\omega$. If $\a$ has cofinality $\omega$ in $L[A]$, then trivially it has  cofinality $\omega$ in $V$. Suppose then $L[A]\models\cof(\a)>\omega$. If $\a<\omega_2^V$, then $L[A]\models \a<\omega_2$, whence $L[A]\models\cof(\a)=\omega_1$. Since $\omega_1^{L[A]}=\omega_1^V$, we obtain  $\cof(\a)=\omega_1$. Suppose therefore $\a\ge\omega_2^V$, but $\cof^V(\a)=\omega$. Note that we can assume $\neg 0^\sharp$, because otherwise $\omega_2^V$ is inaccessible in $L$ already by the general properties of $0^\sharp$. By the Covering Lemma, a consequence of $\neg 0^\sharp$, we have $L[A]\models\cof(\a)\le\omega_1$. Since $L[A]\models\cof(\a)>\omega$, we obtain $L[A]\models\cof(\a)=\omega_1$, and since $\omega_1^{L[A]}=\omega_1$, we have $\cof^V(\a)=\omega_1$. This finishes our proof that $(C^*)^{L[A]}=C^*$. Note that  $L[A]$ satisfies $CH$. On the other hand, we have assumed that there
are $\aleph_2^V$ reals in $C^*$. Thus there are $\aleph_2$ reals in $(C^*)^{L[A]}\subseteq L[A]$, a contradiction.
\end{proof}

\section{Higher order logics}
%
%
%

The basic result about higher order logics, proved in \cite{MR0281603}, is that they give rise to the inner model $\hod$ of hereditarily ordinal definable sets. In this section we show that this result enjoys some robustness, i.e. ostensibly much weaker logics than second order logic still give rise to $\hod$.

\begin{theorem}[Myhill-Scott \cite{MR0281603}]\label{msc}
$C(\sol)=\hod$.
\end{theorem}

\begin{proof}We give the proof for completeness.
We show $\hod\subseteq C(\sol)$. Let $X\in \hod$. There is a first order $\phi(x,\vec{y})$ and ordinals $\vec{\beta}$ such that for all $a$ $$a\in X\iff \phi(a,\vec{\beta}).$$ By Levy Reflection there is an $\alpha$ such that $X\subseteq V_\alpha$ and for all $a\in V_\alpha$
$$a\in X\iff V_\alpha\models\phi(a,\vec{\beta}).$$
Since we proceed by induction, we may assume $X\subseteq C(\sol)$. Let $\gamma$ be such that $X\subseteq L'_{\gamma}$. We can choose $\gamma$ so big that $|L'_{\gamma}|\ge|V_\alpha|$. We show now that $X\in L'_{\gamma+1}$. We give a second order formula $\Phi(x,y,\vec{z})$ such that 
$$X=\{a\in L'_{\gamma}: L'_{\gamma}\models\Phi(a,\alpha,\vec{\beta})\}.$$
We know
$$X=\{a\in L'_{\gamma}: V_\alpha\models\phi(a,\vec{\beta})\}.$$
Intuitively, $X$ is the set of
$a\in L'_{\gamma}$ such that in $L'_{\gamma}$ some $(M,E,a^*,\alpha^*,\vec{\beta^*})\cong (V_\alpha,\in,a,\alpha,\vec{\beta})$ satisfies $\phi(a^*,\vec{\beta^*})$.
Let $\theta(x,y,\vec{z})$ be a second order formula of the vocabulary $\{E\}$ such that 
for any $M$, $E\subseteq M^2$ and $a^*,\alpha^*,\vec{\beta^*}\in M$:
$(M,E)\models\theta(a^*,\alpha^*,\vec{\beta^*})$
iff 
there are an  isomorphism $\pi:(M,E)\cong (V_\delta,\in)$ such that $\pi:(\alpha^*,E)\cong (\delta,\in)$,  and $(V_\delta,\in)\models\phi(\pi(a^*),\pi(\vec{\beta}))$. %

We conclude 
 $X\in L'_{\gamma+1}$
by proving the:
\medskip

\noindent{\bf Claim} The following are equivalent for $a\in L'_{\gamma}$:
\begin{description}
\item[(1)] $a\in X$.
\item[(2)] $L'_{\gamma}\models\exists M,E(\TC(\{a\})\cup\alpha+1\cup\vec{\beta}\cup\{\vec{\beta}\}\subseteq M\wedge(M,E)\models\theta(a,\alpha,\vec{\beta}))\}.$
\end{description}

\noindent $(1)\to(2):$
Suppose  $a\in X$. Thus $V_\alpha\models\phi(a,\vec{\beta})$. Let $M\subseteq L'_{\gamma}$ and $E\subseteq M^2$ such that $\alpha+1,\TC(a),\vec{\beta}\in M$ and there is an isomorphism $$f:(V_\alpha,\in,\alpha,a,\vec{\beta})\cong(M,E,\alpha^*,a^*,\vec{\beta^*}).$$
We can assume $\alpha^*=\alpha$, $a^*=a$ and $\vec{\beta^*}=\vec{\beta}$ by doing a partial Mostowski collapse for $(M,E)$. So then $(M,E)\models\phi(a,\vec{\beta})$, whence $(M,E)\models\theta(a,\alpha,\vec{\beta})$. We have proved (2).
\smallskip

\noindent $(2)\to(1):$
Suppose $M\subseteq L'_{\gamma}$ and $E\subseteq M^2$ such that $\TC(\{a\})\cup\alpha+1\cup\vec{\beta}\cup\{\vec{\beta}\}\subseteq M$ and $(M,E)\models\theta(a,\alpha,\vec{\beta})$. We may assume 
$E\restriction \TC(\{a\})\cup\alpha+1\cup\vec{\beta}\cup\{\vec{\beta}\}=\in\restriction \TC(\{a\})\cup\alpha+1\cup\vec{\beta}\cup\{\vec{\beta}\}$. There is an isomorphism $\pi:(M,E)\cong (V_\alpha,\in)$ such that  $(V_\alpha,\in)\models\phi(\pi(a),\pi(\vec{\beta}))$. But $\pi(a)=a$ and $\pi(\vec{\beta})
=\vec{\beta}$. So in the end  $(V_\alpha,\in)\models\phi(a,\vec{\beta})$. We have proved (1).

\end{proof}

In second order logic $\L^2$ one can quantify over arbitrary subsets of the domain.
A more general logic is obtained as follows: 

\begin{definition}
Let $F$ be any class function on cardinal numbers. The logic $\L^{2,F}$ is like $\L^2$ except that the second order quantifiers range over a domain $M$  over subsets of $M$ of cardinality $\le \kappa$ whenever  $F(\kappa) \le |M|$. 
\end{definition}

Examples of possible functions are $F(\kappa)=0, \kappa$, $\kappa^+$, $2^\kappa$, $\aleph_\kappa$, $\beth_\kappa$, etc.
Note that $\L^2 = \L^{2,F}$ whenever $F(\kappa)\le \kappa$ for all $\kappa$.
The logic $\L^{2,F}$ is  weaker the bigger values $F(\kappa)$ takes on. For example, if $F(\kappa)=2^{2^\kappa}$, the second order variables of $\L^{2,F}$ range over ``tiny" subsets of the universe. Philosophically second order logic is famously marred by the difficulty of imagining how a universally quantified variable could possibly range over {\em all} subsets of an infinite domain. If the universally quantified variable ranges only over ``tiny" size subsets, one can conceivably think that there is some coding device which uses the elements of the domain to code all the ``tiny" subsets.  

Inspection of the proof of Theorem~\ref{msc} reveals that actually the following more general fact holds:

\begin{theorem}\label{ms}
For all $F$: $C(\L^{2,F})=\hod$.
\end{theorem}

Let $\L^{2}_{\kappa}$ denote the modification of $\sol$ in which the second order variables range 
over subsets (relations, functions, etc) of cardinality at most $\kappa$. 

\begin{theorem} Suppose $0^\sharp$ exists. Then
$0^\sharp\in C(\L^{2}_{\kappa})$ 
\end{theorem}

\begin{proof} As in the proof of Theorem~\ref{cof}. \end{proof}

A consequence of  Theorem \ref{ms} is the following:
\medskip

\noindent{\bf Conclusion:} The second order constructible hierarchy $C(\L^2)=\hod$ is unaffected if second order logic is modified in any of the following ways:
\begin{itemize}
\item Extended in any way to a logic definable with hereditarily ordinal definable parameters. This includes third order logic, fourth order logic, etc.

\item Weakened by allowing second order quantification in domain $M$ only over subsets $X$ such that $2^{|X|}\le|M|$.

\item Weakened by allowing second order quantification in domain $M$ only over subsets $X$ such that $2^{2^{|X|}}\le|M|$.

\item Any combination of the above.

\end{itemize}
Thus G\"odel's $\hod=C(\L^2)$ has  some robustness as to the choice of the logic $\L^2$. It is the common feature of the logics that yield $\hod$ that they are able to express quantification over all subsets of some part of the universe the size of which is not a priori bounded. We can perhaps say, that this is the essential feature of second order logic that results in $C(\L^2)$ being $\hod$. What is left out are logics in which one can quantify over, say all countable subsets. Let us call this logic $\L^2_{\aleph_0}$.  Consistently\footnote{Assume $V=L$ and add a Cohen subset $X$ of $\omega_1$. Now code $X$ into $\hod$ with countably closed forcing using \cite{MR0292670}. In the resulting model $C(\L^2_{\aleph_0})=L\ne\hod$.}, $C(\L^2_{\aleph_0})\ne\hod$. Many would call a logic such as  $\L^2_{\aleph_0}$ second order.

Let $\Sigma^1_n$ denote the fragment of second order logic in which the formulas have, if in prenex normal form with second order quantifiers preceding all first order quantifiers, only $n$ second order quantifier alternations, the first second order quantifier being existential.
Note that trivially $C(\Sigma^1_n)=C(\Pi^1_n)$.  Let us write 
$$\hod_n=_{\rm \tiny df}C(\Sigma^1_n).$$ 
The Myhill-Scott proof shows that $\hod_n=\hod$ for $n\ge 2$. What about $\hod_1$?
Note that for all $\beta$ and $A\in \hod_1$:

\begin{itemize}
 \item $\{\a<\beta : \cof^V(\a)=\omega\}\in\hod_1$
 \item $\{(a,b)\in A^2 : |a|^V\le|b|^V\}\in\hod_1$
 \item $\{\a<\beta : \a\mbox{ cardinal in $V$}\}\in\hod_1$
 \item $\{(\a_0,\a_1)\in \beta^2 : |\a_0|^V\le (2^{|\a_1|})^V\}\in\hod_1$
 \item $\{\a<\beta : (2^{|\a|})^V=(|\a|^{+})^V\}\in\hod_1$
 \end{itemize}
These examples show that $\hod_1$ contains most if not all of the inner models considered above. In particular we have:

\begin{lemma}
\begin{enumerate}
\item ${C^*}\subseteq \hod_1$.
\item $C(Q^{\MMa,<\omega}_1)\subseteq \hod_1$ 
\item If $0^\sharp$ exists, then  $0^\sharp\in \hod_1$\end{enumerate}
 \end{lemma}

Naturally, $\hod_1=\hod$ is consistent, since we only need to assume $V=L$. So we focus on  $\hod_1\ne\hod$.

\begin{theorem}\label{hkloiuk}
It is consistent, relative to the consistency of infinitely many weakly compact cardinals that for some $\lambda$: $$\{\kappa <\lambda: \kappa\mbox{ weakly compact (in $V$)}\}\notin \hod_1,$$ and, moreover, $\hod_1=L\ne \hod$.  
\end{theorem}

\begin{proof}
Let us assume $V=L$. Let $\kappa_n, n<\omega$ be  a sequence of  weakly compact cardinals. Let $\mathbb{D}_\delta$ be the forcing notion for adding a Cohen subset of the regular cardinal $\delta$. Let $\lambda=\sup_n \kappa_n $.   We proceed as in \cite{MR495118}. Let $\eta<\kappa$  be two regular cardinals. We denote by  $\mathbb{R}_{\eta,\kappa}$  the Easton support iteration of $\mathbb{D}_\delta$ for  $\eta\leq\delta\leq\kappa$.  The forcing $\mathbb{R}_{\kappa_{n-1}^+,\kappa_n}$, where for $n=0$ we take $\kappa_{-1}=\omega_1$, we denote by $\mathbb{P}_n$. Note that forcing with $\mathbb{P}_n$ preserves the weak compactness of $\kappa_n$. Let $\mathbb{D}_n$ be the name for the forcing $\mathbb{D}_{\kappa_n}$ defined in $V^{\mathbb{P}_n}$. Note that $\mathbb{P}_n\ast \mathbb{D}_{\kappa_n}$ is forcing equivalent to $\mathbb{P}_n$.

Let $\mathbb{Q}$ be the full support product of $\mathbb{P}_n,n<\omega$. Let $V^\ast=V^{\mathbb{Q} }$.

\begin{claim} For every $n<\omega$ the cardinal $\kappa_n$ is weakly compact in $V^\ast$.
\end{claim}

The argument uses the fact that for each $n$ the forcing $\mathbb{Q}$ can be decomposed as $\mathbb{Q}_n\times\mathbb{P}_n\times \mathbb{Q}^n$ where $\mathbb{Q}_n$ has cardinality $\kappa_{n-1}$ and $\mathbb{Q}^n$ is $\kappa_n^+$ closed. Hence $\mathbb{Q}_n$ and $\mathbb{Q}^n$ do not change the weak compactness of $\kappa_n$,  which is preserved by $\mathbb{P}_n$. 

As in \cite{MR495118},  we define in $V^{\mathbb{P}_n}$ a forcing $\mathbb{S}_n$ to be the canonical  forcing which introduces a $\kappa_n$ homogeneous Soulin tree. In particular it kills the weak compactness of $\kappa_n$. Let $\mathbb{T}_n$ be the forcing which introduces a branch through the tree forced by $\mathbb{S}_n$. As in \cite{MR495118} we can show that $ \mathbb{S}_n\ast \mathbb{T}_n$ is forcing equivalent to $\mathbb{D}_{\kappa_n}$. Therefore if we force with $\mathbb{T}_n$ over $V^{\mathbb{P}_n\ast\mathbb{S}_n}$, we regain the weak compactness of $\kappa_n$. Also a generic object for $\mathbb{D}_{\kappa_n}$ introduces a generic object for $\mathbb{T}_n$.

We are going to describe three models $V_1\subseteq V_2\subseteq V_3$. Let first $V_1^\ast=L^\mathbb{Q}$ . Let $G_n$ be the generic filter in $\mathbb{P}_n$, introduced by $\mathbb{Q}$. The model $V_3^\ast$ is the model one gets from $V^*_1$ by forcing over it with the full support product of $\mathbb{D}_{\kappa_n}$. ($\mathbb{D}_n$ is as realized according to $G_n$.). Let $H_n\subseteq \mathbb{D}_n$ be the generic filter introduced by this forcing. Note that $V_3^\ast$ can also be obtained from $L$ by forcing with $\mathbb{Q}$. In particular both in $V_1$ and in $V_3^\ast$ the cardinals $\kappa_n$ are weakly compact for every $n<\omega$. Let $V_3$ be an extension of $V_3^\ast$ by adding a Cohen real $a\subseteq\omega$. Let $\mathbb{A}$ be the Cohen forcing on $\omega$. Then define $V_1=V_1^\ast(a)$. Both $V_1$ and $V_3$ are obtained by forcing over $L$ with $\mathbb{Q}\times \mathbb{A}$ which is a homogenous forcing notion.  Hence $HOD^{V_1}=HOD^{V_3}=L$. Again we did not kill the weak compactness of the cardinals $\kappa_n$.

  Now we define $V_2$. Each $H_n$ introduces a generic filter for the forcing $\mathbb{S}_n$ (As defined according to $G_n$). Let $K_n\subseteq \mathbb{S}_n$ be this generic filter.  We define $$V_2=V_1[a,\langle K_n | n\not\in a \rangle, \langle H_n|n\in a\rangle ].$$

For  $n<\omega$ we define an auxiliary universe $W_n$ as follows: $$W_n=L(a,\langle G_i|i\leq n \rangle, \langle K_i |i\leq n, i\not\in a\rangle, \langle H_i|i\leq n,i\in a \rangle). $$ If $n\in a$ then $W_n$ is obtained from $L$ by a product of $\mathbb{P}_n\ast \mathbb{D}_n$ and some forcings of size $<\kappa_n$. Since $\mathbb{P}_n\ast \mathbb{D}_n$ preserves the weak compactness of $\kappa_n$,  $\kappa_n$ is weakly compact in $W_n$.
If $n\not\in a$ then $K_n$ generates a tree on $\kappa_n$ which is still Souslin in $W_n$. (Small forcings do not change the Souslinity of a tree.), So $\kappa_n$ is not weakly compact in $W_n$. We proved:

\begin{claim} $\kappa_n$ is weakly compact in $W_n$ iff $n\in a$.
\end{claim}

The following claim follows from the standard arguments analysing the power-set of a cardinal $\delta$ under a forcing which is the product of a forcing of size $\mu<\delta$, a forcing of size $\delta$ which is $\mu^+$-distributive, and a forcing which is $\delta^+$-distributive.

\begin{claim} For $n<\omega$ $P(\kappa_n)^{V_2}=P(\kappa_n)^{W_n}$.
\end{claim}

From the last two claims it follows that $$V_2\models a=\{n<\omega | \kappa_n\quad \mbox{is weakly compact} \}.$$ Therefore $a\in HOD^{V_2}$.

The proof of the Theorem will be finished if we show that $HOD_1^{V_2}=L$. 
 For an ordinal $\alpha$ let $L^1_\alpha, L^2_\alpha, L^3_\alpha$ be the $\alpha$-th  step of the construction of $({C}(\Sigma^1_1))^{V_1}$, $({C}(\Sigma^1_1))^{V_1}$,
     $(C(\Sigma^1_1))^{V_3}$ respectively.
\begin{lemma} For every $\alpha$ $L^1_\alpha=L^2_\alpha =L^3_\alpha$.
\end{lemma}
 The proof of the lemma is by induction on $\alpha$ where the cases $\alpha=0$ and $\alpha$ limit are obvious. 
 So given $\alpha$, by the induction assumption on $\alpha$ we can put  $M=L^1_\alpha=L^2_\alpha=L^3_\alpha$. Note that $M\in L$ since $M\in HOD^{V_1}=L$. Let $\Phi(\vec{x})$ be a $\Sigma^1_1$ formula and let $\vec{b}$ be a vector of elements of $M$. 
 \begin{lemma} The following are equivalent
 \begin{enumerate}
   \item  $(M\models \Phi(\vec{b}))^{V_1}$
   \item $(M\models \Phi(\vec{b}))^{V_3}$
   \item $(M\models \Phi(\vec{b}))^{V_2}$
 \end{enumerate}
 \end{lemma}
  
  Without loss of generality, $\Phi(\vec{x})$ has the form $\exists X\Psi(X,\vec{x})$, where $X$ is a second order variable and  all the quantifiers of $\Psi$ are first order. Both $V_1$ and $V_3$ are obtained form $L$ by forcing over $L$ with $\mathbb{Q}\times \mathbb{A}$. This forcing is homogeneous. $M$ and all the elements of the vector $\vec{b}$ are in $L$. So (1) is clearly equivalent to (2). 
  
  Now suppose that  $(M\models \Phi(\vec{b}))^{V_2}$. Let $Z\subseteq M$ be the witness for the existential quantifier of $\Phi$. Then $(M\models \Psi(Z,\vec{b}))^{V_2}$. But all the quantifiers of $\Psi$ are first order,  so $(M\models\Psi(Z,\vec{b}))^{V_3}$. So (3) implies (2),  and hence (1).  For the other direction, if $(M\models \Phi(\vec{b}))^{V_3}$, then we know that $(M\models \Phi(\vec{b}))^{V_1}$. Let $Z\in V_1$ satisfy $M\models \Psi(Z,\vec{b})$. So  $(M\models \Psi(Z,\vec{b}))^{V_2}$, and therefore $(M\models \Phi(\vec{b}))^{V_2}$.
  
  It follows from the lemma that  every $\Sigma^1_1$ formula defines the same subset of $M$ in $V_1$, $V_2$ and  $V_3$. It follows that $L^1_{\alpha+1}=L^2_{\alpha+1}=L^3_{\alpha+1}$. 
  
  This proves the lemma and the theorem.
\end{proof}

The above proof works also with ``weakly compact" replaced by other large cardinal properties, e.g. ``measurable" or ``supercompact".  We can start, for example, with a $\omega$ supercompact cardinals, code each one of them into cardinal exponentiation, detectible by means of $\hod_1$, above all of them, without losing their supercompactness or introducing new supercompact cardinals, and then proceed as in the proof of Theorem~\ref{hkloiuk}. Note that we can also start with a supercompact cardinal and code, using the method of \cite{MR0540771}, every set into cardinal exponentiation, detectible by means of $\hod_1$, without losing the supercompact cardinal. In the final model there is a super compact cardinal while $V=\hod_1$.

We shall now prove an analogue of Theorem~\ref{hkloiuk} without assuming any large cardinals. Let $\oC(\kappa)$ be Cohen forcing for adding a subset for a regular cardinal $\kappa$. Let $R(\kappa)$ be the statement that  there is a bounded subset $A\subseteq \kappa$ and a  set $C\subseteq\kappa$ which is $\oC(\kappa)$-generic over $L[A]$, such that $\P(\kappa)\subseteq L[A,C]$.

\begin{theorem}\label{hkloiukk}
It is consistent, relative to the consistency of {\zfc} that: $$\{n <\omega: R(\aleph_n)\}\notin \hod_1,$$ and, moreover, $\hod_1=L\ne \hod$.  
\end{theorem}

\begin{proof} The proof is very much like the proof of Theorem~\ref{hkloiuk} so we only indicate the necessary modifications.
Let us assume $V=L$. As a preliminary forcing $\oC$ we apply Cohen forcing $\oC_n=\oC(\aleph_n)$ for each $\aleph_n$ (including $\aleph_0$) adding a Cohen subset $C_n\subseteq \aleph_n\setminus\{0,1\}$.  W.l.o.g. $\min(C_{n+1})>\aleph_n$. Let $V_1$ denote the extension. Let $\oP$ be the product forcing in $V_1$ which adds a non-reflecting stationary set $A_n$ to $\kappa_n$, $n\notin C_0$, by means of:
\begin{eqnarray*}
 \oP_n&=&\{p:\g\to 2 : \aleph_{n-1}<\g<\aleph_n, \forall \a<\g(\cof(\a)>\om\to \\
&& \hspace{2cm }\{\beta<\a : p(\beta)=0\}\mbox{ is non-stationary in }\a)\}.
 \end{eqnarray*}
 Let us note that $\oP_n$ is strategically $\aleph_{n-1}$-closed, for the second player can play systematically at limits in such a way that during the game a club is left out of $\{\beta : p(\beta)=0\}$.
 Let $V_2$ denote the extension of $V_1$ by $\oP$. Now 
 $$V_2\models C_0=\{n<\omega : R(\aleph_n)\},$$ for if $n\in C_0$, then $R(\aleph_n)$ holds in $V_2$ by construction, and on the other hand, if $n\notin C_0$, then $R(\aleph_n)$ fails in $V_2$ because one can show with a back-and-forth argument that with $\oP$ and $\oC$ as above, we always have
$V^{\oP}\ne V^{\oC}$.

Let $\oQ$ force in $V_2$ a club into $A_n$, $n\in C_\omega$, by closed initial segments with a last element. The crucial observation now is that
$\oP_n\star\oQ_n$ is the same forcing as $\oC_n$. To see this,
it suffices to find a dense $\aleph_{n-1}$-closed subset of $\oP_n\star\oQ_n$ of cardinality $\aleph_n$. Let $D$ consist of pairs $(p,A)\in\oP_n\star\oC_n$ such that $\dom(p)=\max(A)+1$ and $\forall\beta\in A(p(\beta)=1)$. This set is clearly  $\aleph_{n-1}$-closed.

\end{proof}

\begin{proposition}\label{adpquiowrry}
 If $0^\sharp$ exists, then $0^\sharp\in C(\Delta^1_1)$, hence $C(\Delta^1_1)\ne L$. 
\end{proposition}

\begin{proof}
As Proposition~\ref{cof}.
\end{proof}

\section{Semantic extensions of {\zfc}}

For another kind of application of extended logics in set theory we consider the following concept:

\begin{definition}
Suppose $\L^*$ is an abstract logic. We use ${\zfc}(\L^*)$ to denote the usual ${\zfc}$-axioms in the vocabulary $\{\in\}$ with the modification that the formula $\phi(x,\vec{y})$ in the Schema of Separation $$\forall x\forall x_1...\forall x_n\exists y\forall z(z\in y\leftrightarrow(z\in x\wedge \phi(z,\vec{x})))$$ and the formula $\psi(u,z,\vec{x})$ in the Schema of Replacement 
$$\begin{array}{ll}
\forall x\forall x_1...\forall x_n
(&\hspace{-3mm}\forall u\forall z\forall z'((u\in x\wedge\psi(u,z,\vec{x})\wedge\psi(u,z',\vec{x}))\to z=z')\\
&\to\exists y\forall z
(z\in y\leftrightarrow\exists u(u\in x\wedge \psi(u,z,\vec{x})))).
\end{array}$$
is allowed to be taken from $\L^*$.
\end{definition}

The concept of a a model $(M,E)$, $E\subseteq M\times M$, satisfying the axioms ${\zfc}(\L^*)$ is obviously well-defined. Note that ${\zfc}(\L^*)$ is at least as strong as ${\zfc}$ in the sense that every model of ${\zfc}(\L^*)$ is, a fortiori, a model of ${\zfc}$.

The class of (set) models of ${\zfc}$ is, of course, immensely rich, ${\zfc}$ being a first order theory. If ${\zfc}$ is consistent, we have countable models, uncountable models, well-founded models, non-well-founded models etc. We now ask the question, what can we say about the  models of ${\zfc}(\L^*)$ for various logics $\L^*$? Almost by definition, the inner model $C(\L^*)$ is a class model of ${\zfc}(\L^*)$:
$$C(\L^*)\models {\zfc}(\L^*).$$
But ${\zfc}(\L^*)$ can very well have other models.

\begin{theorem}
A model of ${\zfc}$ is a model of ${\zfc}(\L(Q_0))$ if and only if it is an $\omega$-model.
\end{theorem}

\begin{proof}
Suppose first $(M,E)$ is an $\omega$-model of ${\zfc}$. Then we can eliminate $Q_0$ in $(M,E)$: Given a first order formula $\phi(x,\vec{a})$ with some parameters $\vec{a}$ there is, by the Axiom of Choice, either a one-one function from 
\begin{equation}
\label{q0}
\{b\in M:(M,E)\models\phi(b,\vec{a})\}
\end{equation} onto a natural number of $(M,E)$ or onto an ordinal of $(M,E)$ which is infinite in $(M,E)$. Since $(M,E)$ is an $\omega$-model, these two alternatives correspond exactly to (\ref{q0}) being finite (in V) or infinite (in V). So $Q_0$ has, in $(M,E)$, a first order definition. For the converse, suppose $(M,E)$ is a model of ${\zfc}(\L(Q_0))$ but some element $a$ in $\omega^{(M,E)}$ has infinitely many predecessors in $V$. By using the Schema of Separation, applied to $\L(Q_0)$, we can define the set $B\in M$ of elements $a$ in $\omega^{(M,E)}$ that have infinitely many predecessors in $V$. Hence we can take the smallest element of $B$ in $(M,E)$. This is clearly a contradiction.     
\end{proof}

In similar way one can show that a model of ${\zfc}$ is a model of ${\zfc}(\L(Q_1))$ if and only if it its set of ordinals is $\aleph_1$-like or it has an $\aleph_1$-like  cardinal.

\begin{theorem}
A model of ${\zfc}$ is a model of ${\zfc}(\L(Q^{\MMa}_0))$ if and only if it is well-founded.
\end{theorem}

\begin{proof}
Suppose first $(M,E)$ is a well-founded model of ${\zfc}$. Then we can eliminate $Q^{\MMa}_0$ in $(M,E)$ because it is absolute: The existence of an infinite set $X$ such that every pair from the set satisfies a given first-order formula can be written as the non-well-foundedness of a relation in $M$ and non-well-foundedness is an absolute property in transitive models. For the converse, suppose $(M,E)$ is a model of ${\zfc}(\L(Q_0^{\MMa}))$. Since $Q_0$ is definable from $Q_0^{\MMa}$ we can assume $(M,E)$ is an $\omega$-model and $\omega^{(M,E)}=\omega$. Suppose  some ordinal $a$ in ${(M,E)}$ is non-well-founded. To reach a contradiction it suffices to show that the set of  such $a$ is $\L(Q^{\MMa}_0)$-definable in $(M,E)$. Let $\phi(x,y,z)$ be the first order formula of the language of set theory which says: \begin{itemize}

\item $x=\la x_1,x_2\ra$, $y=\la y_1,y_2\ra$
\item $x_1,x_2<\omega$, $y_1,y_2< a$

\item $x_1\ne x_2$

\item $x_1<x_2\to y_2<y_1$.

\end{itemize}     
Let us first check that $Q_0^{\MMa}xy\phi(x,y,a)$ holds in $(M,E)$. Let $(a_n)$ be a decreasing sequence (in $V$) of elements of $a$. Let $X$ be the set of pairs $\la n,a_n\ra$, where $n<\omega$. By construction, any pair $\la x,y\ra$ in $[X]^2$ satisfies $\phi(x,y,a)$. Thus $Q_0^{\MMa}xy\phi(x,y,a)$ holds in $(M,E)$. For the converse, suppose $Q_0^{\MMa}xy\phi(x,y,b)$ holds in $(M,E)$. Let $Y$ be an infinite set such that every $\la x,y\ra$ in $[Y]^2$ satisfies $\phi(x,y,b)$. Every two pairs in $Y$ have a different natural number as the first component. So we can choose pairs from $Y$ where the first components increase. But then the second components decrease and $b$ has to be non-well-founded.
\end{proof}

\begin{theorem}
A structure is a model of ${\zfc}(\looo)$ if and only if it is isomorphic to a transitive a model $M$ of ${\zfc}$ such that $M^\omega\subseteq M$.
\end{theorem}

\begin{proof}
Suppose first $M$ is a transitive a model of ${\zfc}$ such that $M^\omega\subseteq M$. Then we can eliminate $\Loo$ because the semantics of $\Loo$ is absolute in transitive models and the assumption $M^\omega\subseteq M$ guarantees that all the $\Loo$-formulas of the language of set theory are elements of $M$. For the converse, suppose $(M,E)$ is a model of ${\zfc}(\Loo)$. Since $Q_0$ is definable in $\Loo$, we may assume that $(M,E)$ is an $\omega$-model  and $\omega^{(M,E)}=\omega$. Suppose $(a_n)$ is a sequence (in $V$) of elements of $M$. Let $$\phi(x,y,u_0,u_1,\ldots,z_0,z_1,\ldots)$$ be the $\Loo$-formula
$$\bigwedge_n(x=u_n\to y=z_n).$$ Note that $(M,E)$ satisfies $$\forall x\in\omega\exists y\phi(x,y,0,1,\ldots,a_0,a_1,\ldots).$$If we apply the Schema of Replacement of ${\zfc}(\Loo)$, we get an element $b$ of $M$ which has all the $a_n$ as its elements. By a similar application of the Schema of Separation we get $\{a_n:n\in\omega\}\in M$. Thus $M$ is closed under $\omega$-sequences and in particular it is well-founded.     
\end{proof}

By a similar argument one can see that the only model of the class size theory ${\zfc}(\lio)$ is the class size model $V$ itself. This somewhat extreme example shows that by going far enough along this line eventually gives everything. One can also remark that the class of models of ${\zfc}(\L_{\omega_1\omega_1})$ is exactly the same as the class of models of ${\zfc}(\looo)$. This is because in transitive models $M$ such that $M^\omega\subseteq M$ also the truth of $\L_{\omega_1\omega_1}$-sentences is absolute. So despite their otherwise huge difference, the logics $\L_{\omega_1\omega}$ and $\L_{\omega_1\omega_1}$ do not differ in the current context.

Second order logic is again an interesting case. Note that ${\zfc}(L^2)$ is by no means the same as the so-called second order ${\zfc}$, or ${\zfc}^2$ as it is denoted. We have not changed the Separation and Replacement Schemas into a second order form, we have just allowed second order formulas to be used in the schemas  instead of first order formulas. So, although the models of ${\zfc}^2$ are, up to isomorphism, of the form $V_\kappa$, and are therefore, a fortiori, also models of ${\zfc}(L^2)$, we shall  see below that models of ${\zfc}(L^2)$ need not be of that form.

\begin{theorem}
Assume $V=L$. A structure is a model of ${\zfc}(L^2)$ if and only if it is isomorphic to a  model $M$ of ${\zfc}$ of the form $L_\kappa$ where $\kappa$ is inaccessible.
\end{theorem}

\begin{proof}
First of all, if $V=L$ and $L_\kappa\models {\zfc}$, where $\kappa$ is inaccessible, then trivially $L_\kappa\models {\zfc}(L^2)$. For the converse, suppose $(M,E)\models {\zfc}(L^2)$. Because $Q^{\MMa}_0$ is definable in $L^2$, we may assume $(M,E)$ is a transitive model $(M,\in)$.

We first observe that the model $M$ satisfies $V=L$. To this end,  suppose  $\alpha\in M$ and $x\in M$ is a subset of $\a$. Let $\beta$ be minimal $\beta$ such that  $x\in L_\beta$. There is a binary relation on $\a$, second order definable over $M$, with order type $\beta$. By the second order Schema of Separation this relation is in the model $M$. So $M\models ``{x\in L_\beta }"$. Hence $M\models V=L$. Let $M=L_\a$. It is easy to see that $\a$ has to be an inaccessible cardinal.

\end{proof}

Note that if $0^\#$ exists, then $0^\#$ is in every transitive model of ${\zfc}(L^2)$.
%
%

If there is an inaccessible cardinal $\kappa$ and we add a Cohen real, then ${\zfc}(L^2)$ has a transitive model $M$ which is not of the form $V_\alpha$, namely the $V_\kappa$ of the ground model. By the homogeneity of Cohen-forcing this model is a model of ${\zfc}(L^2)$ but, of course, it is not $V_\kappa$ of the forcing extension. 
This is a consequence of the homogeneity of Cohen forcing.
Note that $M$ is not a model of $\zfc^2$, the {\em second order} $\zfc$, in which the Separation and Replacement Schemas of $\zfc$ are replaced by their second order  versions, making $\zfc^2$ a finite second order theory. Here we have an example where $\zfc^2\ne {\zfc}(L^2)$.

\section{Open Questions}

This topic abounds in open questions. We mention here what we think as the most urgent:

\begin{enumerate}

\item Can ${C^*}$ contain  measurable cardinals? Note that there are no measurable cardinals if $V={C^*}$ (Theorem~\ref{wpoi}).

\item Does ${C^*}$ satisfy CH, if $V$ has  large cardinals? Note that if there are large cardinals then the relativized version $C^*(x)$ of $C^*$ satisfies CH for a cone of reals $x$ (Theorem~\ref{coneresult}).

\end{enumerate}

}


\bigskip

Juliette Kennedy

Department of Mathematics and Statistics

University of Helsinki
\medskip

Menachem Magidor

Department of Mathematics

Hebrew University Jerusalem
\medskip

Jouko V\"a\"an\"anen 

Department of Mathematics and Statistics

University of Helsinki

and

Institute for Logic, Language and Computation

University of Amsterdam

\end{document}